\documentclass[11pt]{article}

\usepackage{mathpazo}
\usepackage{amsfonts}
\usepackage{amsmath}
\usepackage{graphicx}
\usepackage{latexsym}
\usepackage{mathabx}

\usepackage[margin=1.05in]{geometry}
  
\usepackage[T1]{fontenc}
\usepackage{times}
\usepackage{color,graphicx}
\usepackage{array}
\usepackage{enumerate}
\usepackage{amsmath}
\usepackage{amssymb}
\usepackage{amsthm}
\usepackage{pgfplots}
\usepackage{pgf}
\usepackage{tikz}
\usetikzlibrary{patterns}
\usepgfplotslibrary{patchplots} 
\usetikzlibrary{pgfplots.patchplots} 
\pgfplotsset{width=9cm,compat=1.5.1}

\usepackage{enumitem}
\usepackage[english]{babel}
\usepackage{amsthm}
\newtheorem{theorem}{Theorem}

\newtheorem{definition}[theorem]{Definition}
\newtheorem{lemma}[theorem]{Lemma}
\newtheorem{corollary}[theorem]{Corollary}
\newtheorem{remark}[theorem]{Remark}
\newtheorem{proposition}[theorem]{Proposition}

\newtheorem{example}[theorem]{Example}

\newcommand{\be}{\mathbf e}

\newcommand{\bu}{\mathbf u}
\newcommand{\bv}{\mathbf v}

\newcommand{\bw}{\mathbf w}

\newcommand{\zero}{\mathbf 0}
\newcommand{\supp}{{\Phi}}
\newcommand\spec{{\mathrm{spec}}}

\usepackage{xcolor}
\usepackage{makeidx}
\usepackage[pdftex,linktocpage=true,colorlinks,citecolor=blue,linkcolor=blue,pagebackref]{hyperref}
\usepackage{hyperref}
\usepackage{cleveref}
\usepackage[alphabetic,backrefs]{amsrefs}

\def\R{\mathbb{R}}
\def\Q{\mathbb{Q}}
\def\Z{\mathbb{Z}}

\begin{document}










\title{Sedentary quantum walks on bipartite and planar graphs}

\author{Karen Meagher\textsuperscript{1}, Hermie Monterde\textsuperscript{2}}

\maketitle

\begin{abstract}
If a quantum walk starting on a vertex tends to stay at home, then that vertex is said to be sedentary. We prove that almost all planar graphs and almost all trees contain at least two sedentary vertices for any assignment of edge weights --- a result that suggests vertex sedentariness is a common phenomenon in trees and planar graphs. For weighted bipartite graphs, we show that a vertex is not sedentary whenever 0 does not belong to its eigenvalue support. Consequently, each vertex in a nonsingular weighted bipartite graph is not sedentary, a stark contrast to weighted trees and weighted planar graphs. A corollary of this result is that every vertex in a bipartite graph with a unique perfect matching is not sedentary for any assignment of edge weights. We also construct new families of weighted bipartite graphs with sedentary vertices using the bipartite double, subdivision operation, and corona product. Finally, we show that unweighted paths and unweighted even cycles contain no sedentary vertices.
\end{abstract}

	

\addtocounter{footnote}{1}
\footnotetext{Department of Mathematics and Statistics, University of Regina, SK, Canada S4S 0A2, karen.meagher@uregina.ca}
\addtocounter{footnote}{1}
\footnotetext{Department of Mathematics and Statistics, University of Regina, SK, Canada S4S 0A2,
hermie.monterde@uregina.ca (corresponding author)}
\medskip

\section{Introduction}

Let $A=A(X)$ be the adjacency matrix of a graph $X$. A \textit{continuous-time quantum walk} on $X$ is determined by the complex symmetric unitary matrix 
\begin{equation*}
    U(t)=e^{itA}\quad t\in\R,
\end{equation*}
called the \textit{transition matrix} of $X$. The matrix $A$ is called the \textit{Hamiltonian} of the quantum walk. In \cite{SedQW}, the notion of a sedentary vertex in a quantum walk was introduced. In \cite{Monterde2023}, it was formalized: vertex $u$ in $X$ is sedentary if $|U(t)_{u,u}|$ is bounded away from 0 for all $t>0$. Simply put, a vertex in a graph is sedentary if its associated quantum state tends to stay at that vertex. It was shown that Cartesian products preserve vertex sedentariness, vertices with at least two twins are sedentary, and there are strongly cospectral vertices that are sedentary. Building on these results, a subsequent work \cite{monterde2023new} by the second author characterized sedentary twin vertices, proved that direct products and joins preserve vertex sedentariness under mild conditions, and determined all complete bipartite and threshold graphs that are sedentary. It was also shown that a vertex with a twin is either sedentary or admits pretty good state transfer, which is an interesting observation owing to the fact that sedentariness and pretty good state transfer are mutually exclusive quantum phenomena. However, it turns out that this dichotomous property for twin vertices does not hold in general, as demonstrated by infinite families of twin-free graphs whose every vertex exhibits neither of these two properties. Furthermore, a characterization of vertex sedentariness was given in terms of pretty good group state transfer, a new type of quantum transport introduced in \cite{monterde2023new} and generalizes the notion of group state transfer \cite{Brown2021}.

In this paper, we report further progress on the theory of vertex sedentariness, with emphasis on bipartite graphs, trees, and planar graphs. In Section \ref{Sec:st}, we review a few basic properties of quantum state transfer on graphs. In Section \ref{sec:bip} we show that a vertex in a bipartite graph is not sedentary whenever 0 is not in its eigenvalue support (Theorem \ref{thm:bipartite}). Consequently, a nonsingular bipartite graph has no sedentary vertex (Corollary \ref{cor:intbipgraphs}). A corollary of this result is that a weighted bipartite graph with a unique perfect matching has no sedentary vertices for any assignment of edge weights (Theorem \ref{thm:bipeveneval}). In Section \ref{sec:bip1}, we provide sufficient number-theoretic conditions such that a vertex in a bipartite graph is not sedentary whenever $0$ belongs to the eigenvalue support of that vertex. We deal with trees and planar graphs in Section \ref{Sec:tree}, where we prove that almost all connected planar graphs (resp., trees) contain at least two sedentary vertices for any assignment of edge weights (Corollary \ref{planar}) (resp., Corollary \ref{alltrees}). We complement this result by providing infinite families of trees and planar graphs that do not have sedentary vertices (see Corollaries \ref{subdstar} and \ref{c6}). We also provide new constructions of graphs with sedentary vertices using the bipartite double in Section \ref{Sec:bipdoub}, the subdivision operation in Section \ref{Sec:subd}, the corona product in Section \ref{Sec:corona}, and the Cartesian and rooted products in Section \ref{Sec:cons1}. Section \ref{Sec:bip2} explores vertex sedentariness in bipartite graphs with few distinct eigenvalues. Furthermore, we show that unweighted paths and unweighted even cycles do not contain sedentary vertices in Section \ref{Sec:paths}. Finally, open questions are presented in Section \ref{Sec:open}. The remainder of the present section is allotted to basic definitions and notation relevant to the rest of the paper.

Unless otherwise stated, we assume that our graphs are simple (loopless), connected, weighted and undirected. Let $X$ be a graph with vertex set $V(X)$ and edge set $E(X)$. All edges in $X$ are assumed to have nonzero weights. We say that $X$ is \textit{unweighted} if all edge weights in $X$ are equal to one. For a given assignment of weights on $X$, $A=A(X)$ denotes the \textit{adjacency matrix} of $X$ defined entry-wise as $A_{u,v}=\omega_{u,v}$ if $u$ and $v$ are adjacent, and $A_{u,v}=0$ otherwise. Here, $\omega_{u,v}$ is the weight of the edge $\{u,v\}$. We say that $X$ is a \textit{nonsingular} graph if $A(X)$ is nonsingular. We denote the set of distinct eigenvalues of $A(X)$ by $\spec(X)$. Since $X$ is connected, $A$ has at least two distinct eigenvalues. Moreover, 
we may write $A$ in its spectral decomposition as
\begin{equation}
\label{specdec}
A=\sum_{\lambda\in\spec(X)}\lambda E_{\lambda},
\end{equation}
where $E_{\lambda}$ is the orthogonal projection matrix onto the eigenspace of $A$ associated with $\lambda$. Note that the matrices $E_\lambda$ are all real, idempotent, pairwise multiplicatively orthogonal, and sum to the identity matrix. Using (\ref{specdec}), we may write the transition matrix as
\begin{equation}
\label{specdecU}
U(t)=\sum_{\lambda\in\spec(X)}e^{it\lambda}E_{\lambda}.
\end{equation}
The \textit{eigenvalue support} of vertex $u$ in $X$, denoted $\supp_{\be_u}$, is the set 
\begin{equation*}
\supp_{\be_u}=\{\lambda\in\spec(X):E_{\lambda}\be_u\neq \textbf{0}\}.
\end{equation*}
Two vertices $u$ and $v$ in $X$ are said to be \textit{cospectral} if $(E_{\lambda})_{u,u}=(E_{\lambda})_{v,v}$ for all $\lambda\in\supp_{\be_u}$. Note that if $u$ and $v$ are cospectral, then $U(t)_{u,u}=U(t)_{v,v}$ for all $t\in\R$. 

For a simple unweighted graph $X$ with vertex $u$, let $N_X(u)$ denote the neighborhood of $u$ in $X$. Vertices $u$ and $v$ in $X$ are \textit{twins} if
\begin{equation*}
N_X(u)\backslash \{v\}=N_X(v)\backslash \{u\}.
\end{equation*}
A maximal subset $W$ of $V(X)$ is a \textit{twin set} in $X$ if every pair of vertices in $W$ are twins. In this case, the subgraph of $X$ induced by $W$ is either an empty graph or a clique. From \cite[Theorem 15]{monterde2023new}, $\lambda \in \{0,-1\}$ is an eigenvalue for $A(X)$ whose associated orthogonal projection matrix is given by
\begin{equation}
\label{twins}
E_\lambda=\begin{bmatrix}
     I_{|W|}-\frac{1}{|W|}J_{|W|} & O\\
    O & O
     \end{bmatrix}+F,
\end{equation}
where $\lambda=-1$ if and only if the subgraph induced by $W$ is a clique, and $F$ is an orthogonal projection matrix such that $F(\be_a-\be_b)=0$ for all $a,b\in W$ and $F$ is absent if $\operatorname{rank}E_\lambda=|W|-1$. We now combine Corollary 12 and Theorem 15(2) in \cite{monterde2023new} into one lemma.

\begin{lemma}
\label{monterde}
Let $W$ be a twin set in $X$. Then the vertices in $W$ are either all sedentary, or all involved in pretty good state transfer with each other. Moreover, if $|W|\geq 3$, then each vertex in $W$ is sedentary.
\end{lemma}

We also state the following theorem due to Kronecker \cite[Chapter 3]{levitan1982almost}.

\begin{theorem}
\label{thm:kronecker}
Let $\lambda_1,\ldots,\lambda_r$ and $\zeta_1,\ldots,\zeta_r$ be real numbers. For each $\epsilon>0$, the system of inequalities
\begin{equation}
\label{ineq}
|\lambda_jy-\zeta_j|<\epsilon\quad \pmod{2\pi},\quad j=1,2,\ldots,r
\end{equation}
has a solution $y$ if and only if for all integers $\ell_1,\ldots,\ell_r$,
\begin{equation*}
\displaystyle\sum_{j=1}^r\ell_j\lambda_j=0\quad \text{implies that} \quad \displaystyle\sum_{j=1}^r\ell_j\zeta_j\equiv 0 \ \pmod{2\pi}.
\end{equation*}
\end{theorem}

\begin{remark}
If the eigenvalues $\lambda_j$ are linearly independent over $\Q$ or each $\zeta_j\equiv 0 \pmod{2\pi}$, then (\ref{ineq}) has a solution.
\end{remark}

\section{State transfer}\label{Sec:st}

In this section, we review a few basic properties of state transfer on graphs. Let $u$ and $v$ be two vertices in a graph $X$. We say that $u$ and $v$ admit \textit{perfect state transfer} (PST) in $X$ if $|U(t_0)_{u,v}|=1$ for some $t_0>0$. A relaxation of PST is \textit{pretty good state transfer} (PGST), which happens if there exists a sequence $\{\tau_k\}$ such that $\lim_{k\rightarrow\infty}|U(t_k)_{u,v}|=1$. We say that $u$ is \textit{periodic} in $X$ if $|U(t_0)_{u,u}|=1$ for some $t_0>0$. For a periodic vertex $u$, the smallest $t_0>0$ such that $|U(t_0)_{u,u}|=1$ is called its \textit{minimum period}, this value is usually denoted by $\rho$. A set $\mathcal{S}\subset \mathbb{R}$ with at least two elements satisfies the \textit{ratio condition} if 
\begin{equation*}
\frac{\lambda-\theta}{\mu-\eta}\in\Q 
\end{equation*}
for any $\lambda,\theta,\mu,\eta\in \mathcal{S}$ with $\mu\neq \eta$. Note that if $|\mathcal{S}|=2$, then $\mathcal{S}$ automatically satisfies the ratio condition. The following characterization of periodicity of a vertex due to Godsil \cite{Godsil2011}.

\begin{theorem}
\label{thm:rc}
Vertex $u$ is periodic in $X$ if and only if $\supp_{\be_u}$ satisfies the ratio condition. 
\end{theorem}

Next, we state the formal definition of a sedentary vertex \cite{Monterde2023}.

\begin{definition}
\label{def}
Let $X$ be a graph and $u\in V(X)$. Vertex $u$ is \textit{$C$-sedentary} in $X$ if
\begin{equation}
\label{inf}
\inf_{t>0}\lvert U(t)_{u,u}\rvert\geq C
\end{equation}
for some constant $0<C\leq 1$. If the $C$ in (\ref{inf}) is attained by $\lvert U(t)_{u,u}\rvert$ at some time $t=t_0$, then vertex $u$ is $C$-sedentary in $X$ with equality at $t=t_0$. Vertex $u$ is \textit{not sedentary} in $X$  if $\inf_{t>0}\lvert U(t)_{u,u}\rvert=0$. Further, $X$ is $C$-sedentary (resp., not sedentary) if each vertex in $X$ is $C$-sedentary (resp., not sedentary).
\end{definition}

\begin{figure}[h!]
	\begin{center}
		\begin{tikzpicture}
		\tikzset{enclosed/.style={draw, circle, inner sep=0pt, minimum size=.2cm}}	   
	   \node[enclosed, label={left, yshift=0cm: $u$}] (w_1) at (1.5,3.7) {};
	    \node[enclosed, fill=black] (w_2) at (0,2.5) {};
		\node[enclosed, fill=black] (w_3) at (1.5,2.5) {};
		\node[enclosed, fill=black] (w_4) at (3,2.5) {};
		\node[enclosed, fill=black] (w_6) at (5,2.5) {};
		\draw (w_1) -- (w_3);
		\draw (w_2) -- (w_3);
		\draw (w_4) -- (w_3);
		\draw (w_4) -- node[above] {$\overbrace{}^{n-2\ \text{vertices}}$} (w_6)[dashed];
		\end{tikzpicture}
	\end{center}
	\caption{The graph $G_n$ with pendent vertex $u$}\label{yay1}
\end{figure}
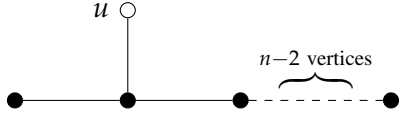
If $C$ is not important, then we simply say \textit{sedentary} in lieu of $C$-sedentary. Sedentariness of vertex $u$ implies that $\lvert U(t)_{u,u}\rvert$ is bounded away from $0$. Physically speaking, this means that the quantum state initially at vertex $u$ tends to stay at $u$. The most common example of a sedentary graph is the complete graph $K_n$, which is $\big(1-\frac{2}{n}\big)$-sedentary for all $n\geq 3$ with equality at $t=\frac{\pi}{n}$. We also note that if $u$ and $v$ are cospectral vertices, then $u$ is sedentary if and only if $v$ is.

Our main motivation for studying vertex sedentariness comes from the fact that sedentariness and PGST are mutually exclusive types of quantum transport \cite[Prop 23]{Monterde2023}. Thus, by identifying families of graphs that are sedentary, we have essentially determined qubit network topologies that do not allow for the accurate transmission of quantum information. We also propose a scheme where sedentariness can potentially be used to control the timing of state transfer, whereby we hold a quantum state at a vertex until it is ready to be transferred and held again at the receiving vertex. Suppose $X$ is an unweighted graph that contains a $C$-sedentary vertex $u$ for large enough $C$, and that $X$ admits an assignment of edge weights such that vertices $u$ and $v$ admit PST under this weighting. Let $\hat{X}$ denote the weighted graph where $u$ and $v$ admit PST. We start by initializing the quantum walk on $X$ as an unweighted graph (i.e. the Hamiltonian is $A(X)$). In this case, vertex $u$ is $C$-sedentary, and hence stores the quantum state assigned at vertex $u$. When the quantum state at vertex $u$ is ready to be transferred, one may switch the Hamiltonian to $A(\hat{X})$ to allow for PST between $u$ and $v$ (note that the graph remains the same despite the change in Hamiltonian). In addition, if $u$ and $v$ are cospectral vertices relative to $A(X)$, then we may switch the Hamiltonian back to $A(X)$ once the transfer is completed. In this case vertex $v$ becomes $C$-sedentary, and hence stores the quantum state originally assigned at vertex $u$. By way of example, we may take $X=K_n$ which is sedentary, and then carefully assign weights so that the resulting weighted complete graph admits PST (see \cite{brow} for examples of signed of complete graphs admitting PST). This scheme was adopted from the work of Shi and Wong \cite{shi2026self}, where instead of using simple graphs, the authors used the addition of loops with large weights to induce sedentariness (and thus hold a quantum state at a vertex) and then relied on unweighted graphs with PST to perform the transfer. Using the same argument in \cite{shi2026self}, we may interpret our scheme as a form of quantum memory where the same quantum state is used for `storage' and `flying' mode.

The concept of a sedentary vertex was first introduced by Godsil in \cite{SedQW}, although his notion of a $C$-sedentary vertex was defined relative to a family of graphs $\mathcal{F}$, where $C$ was a function of the number of vertices of $G\in \mathcal{F}$ such that $C(n)\rightarrow 1$ as $n\rightarrow\infty$. One such example is the family of complete graphs $K_n$ for all $n\geq 3$ with $C(n)=1-\frac{2}{n}$. However, this definition was a bit restrictive, since there are families of graphs that contain $C$-sedentary vertices (in the sense of Definition \ref{def}), but $C(n)\rightarrow 0$ as $n\rightarrow\infty$. Such families are called \textit{quasi-sedentary} \cite{Monterde2023}. An example of a quasi-sedentary family is the family $G_n$ of graphs obtained from an odd path $P_n$ by adding a pendent vertex (say vertex $u$) attached to a neighbor of an end vertex (see Figure \ref{yay1}). In this family, vertex $u$ is $\frac{1}{n}$-sedentary in $G_n$ for all odd $n\geq 3$ \cite[Example 2]{monterde2023new}, in which case $C(n)=\frac{1}{n}\rightarrow 0$ as $n\rightarrow\infty$. We note in passing that the sedentary concept introduced by Godsil is useful in the above scheme where a large enough value of $C$ is needed for quantum memory to be robust, while the formalization of sedentariness due to Monterde allowed for the identification of more families of graphs that do not admit PST and PGST. Furthermore, those graphs that satisfy Godsil's notion of sedentariness (e.g.~the complete graphs) mostly not admit local $\epsilon$-uniform mixing, a type of quantum transport that requires $|U(t)_{u,u}|$ to be close to $1/\sqrt{|V(X)|}$ for appropriate choices of $t$  \cite{monterde2026local}. This is in contrast to some graphs that satisfy a more relaxed type of sedentariness (see Examples 46 and 47 in \cite{Monterde2023} for quasi-sedentary families that admit local $\epsilon$-uniform mixing).

The following result is taken from Lemmas 9 and 12 \cite{Monterde2023}.

\begin{lemma}
\label{lem:persed}
Let $u$ be a periodic vertex in $X$. Then $u$ is $C$-sedentary if and only if $U(t)_{u,u}\neq 0$ for all $t\in[0,\rho]$. In this case, $C=\min_{t\in[0,\rho]} |U(t)_{u,u}|$ and $u$ is $C$-sedentary with equality at some $t_0\in [0,\rho]$ .
\end{lemma}

Lemma \ref{lem:persed} implies that for a periodic sedentary vertex $u$, we can find the largest $C$ such that $u$ is $C$-sedentary and this $C$ is attained by $\lvert U(t)_{u,u}\rvert$ at some time $t=t_0$. However, this is not necessarily the case for vertices that are not periodic. For example, vertex $u$ in the graph $G_5$ in Figure \ref{yay1} is not periodic because 
\begin{equation*}
\supp_{\be_u}=\left\{0,\pm\sqrt{(5\pm\sqrt{5})/2}\right\}
\end{equation*}
does not satisfy the ratio condition in Theorem \ref{thm:rc}. Nonetheless, vertex $u$ is $C$-sedentary where $C=\frac{1}{5}$. While this value $C$ is not attained by $\lvert U(t)_{u,u}\rvert$ for any time $t$, the arguments in \cite[Example 2]{monterde2023new} imply that there is a sequence $\{t_k\}$ such that $\displaystyle\lim_{k\rightarrow\infty}\lvert U(t_k)_{u,u}\rvert=C$. Hence, we have that $\inf_{t>0}\lvert U(t)_{u,u}\rvert=C$.

Combining \cite[Corollary 13(1)]{Monterde2023} and \cite[Lemma 9]{monterde2023new} yields the first known sufficient condition for sedentariness, which states that $u$ is sedentary in $X$ whenever the $u$th diagonal entry of a single orthogonal projection matrix is sufficiently large. Here we use the notation $\nu_2(a)$ for the largest power of $2$ that divides an integer $a$.

\begin{theorem}
\label{thm:sed}
Suppose $\theta\in \supp_{\be_u}$. The following hold.
\begin{enumerate}
\item If $(E_{\theta})_{u,u}>\frac{1}{2}$, 
then $u$ is sedentary with $|U(t)_{u,u}|\geq 2(E_{\theta})_{u,u}-1$. 
\item Let $(E_{\theta})_{u,u}=\frac{1}{2}$.
Then $u$ is not sedentary if and only if for any tuple of integers $(\ell_\lambda)_{\lambda\in\supp_{\be_u}\backslash\{\theta\}}$,
\begin{equation}
\label{kr}
\displaystyle\sum_{\lambda\in\supp_{\be_u}\backslash\{\theta\}}\ell_\lambda(\lambda-\theta)=0\quad  \text{implies that} \quad \displaystyle\sum_{\lambda\in\supp_{\be_u}\backslash\{\theta\}}\ell_\lambda\quad \text{is even}.
\end{equation}
Additionally, if $u$ is periodic and $\supp_{\be_u}$ is closed under taking algebraic conjugates, then $u$ is not sedentary if and only if the values of $\nu_2(\frac{\lambda-\theta}{\sqrt{\Delta}})$ are equal for all $\lambda\in\supp_{\be_u}\backslash\{\theta\}$ (if $\supp_{\be_u}\subseteq\Z$ then $\Delta=1$, otherwise $\Delta$ is a square-free integer, greater than 1, that can be determined from $ \supp_{\be_u}$).

\end{enumerate}
\end{theorem}

Let $m_\theta$ denote the multiplicity of $\theta\in\spec(X)$.

\begin{corollary}
\label{cor8} 
If $m_\theta>\frac{n}{2}$, or $m_\theta=\frac{n}{2}$ and $(E_\theta)_{u,u}<\frac{1}{2}$ for some $u\in V(X)$, then $X$ has a sedentary vertex.
\end{corollary}

\begin{proof}
Let $\theta\in \supp_{\be_u}$. Since the trace of $E_{\theta}$ equals $m_\theta$, we get $\frac{1}{n}\sum_{u\in V(X)}(E_{\theta})_{u,u}=\frac{m_\theta}{n}$. So if $m_\theta>\frac{n}{2}$, then $\frac{1}{n}\sum_{u\in V(X)}(E_{\theta})_{u,u}>\frac{1}{2}$. In this case,  $(E_{\theta})_{v,v}>\frac{1}{2}$ for some $v\in V(X)$.
Now, if $m_\theta=\frac{n}{2}$, then $\frac{1}{n}\sum_{u\in V(X)}(E_{\theta})_{u,u}=\frac{1}{2}$. In this case, if $(E_\theta)_{u,u}<\frac{1}{2}$ for some $u\in V(X)$, then it must be that $(E_\theta)_{v,v}>\frac{1}{2}$ for some vertex $v\neq u$. Otherwise, $\frac{1}{n}\sum_{u\in V(X)}(E_{\theta})_{u,u}<\frac{1}{2}$, a contradiction. In both cases,  $(E_{\theta})_{v,v}>\frac{1}{2}$ for some $v\in V(X)$. Applying Theorem \ref{thm:sed}(1) yields the desired result.
\end{proof}

\section{Nonsingular bipartite graphs}
\label{sec:bip}

A weighted bipartite graph $X$ has adjacency matrix of the form
\begin{equation*}
A=\begin{bmatrix}
     O & B\\
    B^T & O
     \end{bmatrix},
\end{equation*}
for some rectangular real matrix $B$. It is known that $\lambda$ is an eigenvalue of $A$ with eigenvector $\begin{bmatrix} \bu \\  \bv\end{bmatrix}$ if and only if $-\lambda$ an eigenvalue of $A$ with eigenvector $\begin{bmatrix} \bu \\  -\bv\end{bmatrix}$. The following is immediate.

\begin{lemma}
\label{lem:bipesupp}
For any vertex $u$ of a weighted bipartite $X$, $\supp_{\be_u}$ is closed under multiplication by $-1$ and 
\begin{equation*}
(E_{\theta})_{u,u}=(E_{-\theta})_{u,u}\quad \text{for all $\theta\in\spec(X)\backslash\{0\}$}.
\end{equation*}
\end{lemma}

\begin{proposition}
\label{prop11}
Let $X$ be a weighted bipartite graph. If $\theta\in \supp_{\be_u}\backslash\{0\}$, then $(E_\theta)_{u,u}\leq \frac{1}{2}$. Additionally, if $0\in\supp_{\be_u}$, then $(E_\theta)_{u,u}<\frac{1}{2}$.
\end{proposition}

\begin{proof}
Let $\theta\in \supp_{\be_u}\backslash\{0\}$. If $(E_\theta)_{u,u}>\frac{1}{2}$, then $(E_\theta)_{u,u}+(E_{-\theta})_{u,u}>1$ by Lemma \ref{lem:bipesupp}, a contradiction. So $(E_\theta)_{u,u}\leq \frac{1}{2}$. If we further assume that $0\in\supp_{\be_u}$, then $(E_{0})_{u,u}>0$. Since $(E_{0})_{u,u}+(E_\theta)_{u,u}+(E_{-\theta})_{u,u}\leq 1$, Lemma \ref{lem:bipesupp} implies that $(E_{0})_{u,u}=0$ whenever $(E_\theta)_{u,u}=\frac{1}{2}$, a contradiction. Thus, $(E_\theta)_{u,u}<\frac{1}{2}$.
\end{proof}

From the above proposition, the only eigenvalue $\theta\in\spec(X)$ for which $(E_\theta)_{u,u}>\frac{1}{2}$ is possible is $\theta=0$. For this reason, we focus our analysis on the zero eigenvalue of bipartite graphs.

Lemma \ref{lem:bipesupp} and (\ref{specdecU}) give us the diagonal entries of $U(t)$ whenever $X$ is bipartite.

\begin{lemma}
\label{lem:bipartite}
Let $X$ be a weighted bipartite graph with vertex $u$ and let $f(t)=\sum_{0<\lambda\in \supp_{\be_u}}\cos(\lambda t)(E_{\lambda})_{u,u}$. If $\theta=0\in\supp_{\be_u}$, then $U(t)_{u,u}=(E_0)_{u,u}+2f(t)$. Otherwise, $U(t)_{u,u}=2f(t)$.
\end{lemma}

Thus, for a bipartite graph $X$, each diagonal entry of $U(t)$ is a continuous real-valued function. Thus, to prove that vertex $u$ is not sedentary, it suffices to show that there is some $t_0\in\R$ such that $U(t_0)_{u,u}<0$.

By virtue of Proposition \ref{prop11}, Theorem \ref{thm:sed}(1) is not useful for a bipartite graph with vertex $u$ such that $0\notin\supp_{\be_u}$. It turns out that such vertices are in fact not sedentary as our next result shows.

\begin{theorem}
\label{thm:bipartite}
If $X$ is a weighted bipartite graph with vertex $u$ such that $0\notin\supp_{\be_u}$, then $u$ is not sedentary.
\end{theorem}

\begin{proof}
Let $X$ be a bipartite graph with vertex $u$ such that $0\notin \supp_{\be_u}$. From Lemma \ref{lem:bipartite}, we have $$U(t)_{u,u}=2\sum_{0<\lambda\in\supp_{\be_u}}\cos(\lambda t)(E_{\lambda})_{u,u}.$$ 
Let $\Lambda=\min_{0<\lambda\in\supp_{\be_u}}\lambda$. 
For each sufficiently small $\epsilon>0$, consider the system of inequalities $|\lambda t|<\Lambda\epsilon\pmod{2\pi}$, for all $0<\lambda\in\supp_{\be_u}$. Applying Theorem \ref{thm:kronecker} with $\zeta_j=0$ for each $j$ yields a solution $t_\epsilon$ to this system of inequalities. Since $\epsilon$ is small, we have
\begin{equation*}
\sin(|\lambda  t_\epsilon|)<\sin(\Lambda\epsilon)<\Lambda\epsilon.
\end{equation*}
Now, because $\cos(\lambda t)$ is continuous over $\mathbb{R}$ for each $0<\lambda\in\supp_{\be_u}$, we get
$$\displaystyle\int_{0}^{|t_\epsilon|}\cos(\lambda t)dt=\frac{1}{\lambda}\sin(\lambda t)\\ \bigg|_{t=0}^{t=|t_\epsilon|}
=\frac{1}{\lambda}\sin(|\lambda t_\epsilon|) <\frac{1}{\lambda}\sin(\Lambda\epsilon)<\frac{\Lambda\epsilon}{\lambda}\leq \epsilon.$$
Consequently,
$$\int_{0}^{|t_\epsilon|} U(t)_{u,u} dt= 2\sum_{0<\lambda\in\supp_{\be_u}}\left(\int_{0}^{|t_\epsilon|}\cos(\lambda t)dt\right)(E_{\lambda_j})_{u,u}< 2\epsilon\sum_{0<\lambda\in\supp_{\be_u}}(E_{\lambda})_{u,u}=\epsilon.$$ 
Equivalently, there exists a sequence of times $\{t_k\}\subseteq\R^+$ such that 
\begin{equation*}
\displaystyle\lim_{k\rightarrow\infty}\int_{0}^{t_k} U(t)_{u,u} dt=0.
\end{equation*}
Finally, since $U(t)_{u,u}$ is continuous over $\mathbb{R}$, the Mean Value Theorem for integrals implies that for each $k$, there exists a $t_k'\in [0,t_k]$ such that
\begin{equation*}
U(t_k')_{u,u}=\displaystyle\frac{1}{t_k}\int_{0}^{t_k} U(t)_{u,u} dt
\end{equation*}
Taking the limit of this equation as $k\rightarrow\infty$ yields
$\lim_{k\rightarrow\infty}U(t_k')_{u,u}=0$. Thus, $u$ is not sedentary.
\end{proof}

We also note that Theorem \ref{thm:bipartite} also follows from \cite[Lemma 8.9.1]{Coutinho2021}.

From Theorem \ref{thm:bipartite}, the following result is immediate.

\begin{corollary}
\label{cor:intbipgraphs}
If $X$ is a nonsingular weighted bipartite graph, then $X$ is not sedentary.
\end{corollary}

Note that $X$ is a nonsingular weighted bipartite graph if and only if $A(X)$ has an even number of distinct eigenvalues. Hence, we may restate Corollary \ref{cor:intbipgraphs} as follows.

\begin{corollary}
\label{evenevals}
If $X$ is a weighted bipartite graph with an even number of distinct eigenvalues, then $X$ is not sedentary.
\end{corollary}

\section{Singular bipartite graphs}\label{sec:bip1}

Next, we consider the case $0\in \supp_{\be_u}$. Since Theorem \ref{thm:sed} handles the case $(E_{0})_{u,u}\geq \frac{1}{2}$, our investigation in this section will focus on the case $(E_{0})_{u,u}<\frac{1}{2}$. We let $\R^+$ denote the set of positive real numbers.

Let $u$ be a vertex in a bipartite graph such that $0\in \supp_{\be_u}$ and $(E_{0})_{u,u}<\frac{1}{2}$. Since $(E_{0})_{u,u}<\frac{1}{2}$, Lemma \ref{lem:bipesupp} and the fact that $\sum_{\lambda\in\supp_{\be_u}}(E_{\lambda})_{u,u}=1$ imply that we can find a nonempty subset $S$ of $\supp_{\be_u}\cap\R^+$ such that $\sum_{\lambda\in S}(E_{\lambda})_{u,u}\geq\frac{1}{4}$, in which case, we may define
\begin{equation}
\label{S}
\text{$\alpha:=\sum_{\lambda\in S}(E_{\lambda})_{u,u}$.}
\end{equation}
Given this assumption, we establish number-theoretic conditions that rule out sedentariness of vertex $u$.

\begin{theorem}
\label{bip0}Let $X$ be a weighted bipartite graph with vertex $u$ such that $0\in \supp_{\be_u}$. Consider the set $S$ in (\ref{S}) and define
\begin{equation*}
\zeta:=\arccos \left(-\frac{1}{4\alpha}\right)\in\left(\frac{\pi}{2},\pi\right].
\end{equation*}
Given a tuple of integers $(\ell_\lambda)_{\lambda\in S}$ and a subset $S'$ of $S$ (possibly empty), let 
\begin{equation*}
D_{S'}:=\sum_{\lambda\in S\backslash S'}\ell_{\lambda}-\sum_{\lambda\in S'}\ell_{\lambda}, 
\end{equation*}
where $\sum_{\lambda\in S'}\ell_{\lambda}$ is absent if $S'=\varnothing$. If for each tuple of integers $(\ell_\lambda)$ whose components are not all zero, 
\begin{equation}
\label{eq:linind}
\displaystyle\sum_{\lambda\in S}\lambda\ell_{\lambda}=0
\end{equation}
implies that (i) $\zeta=\frac{a\pi}{b}$ for some integers $a$ and $b$ such that $\operatorname{gcd}(a,b)=1$ and (ii) there is a subset $S'$ of $S$ (possibly empty) such that $b$ divides $D_{S'}$ and $\nu_2\left( D_{S'}\right)>\nu_2(b)-\nu_2(a)$, then $u$ is not sedentary.
\end{theorem}

\begin{proof}
Consider the set $S$ in our assumption and let $S'$ be a subset  of $S$ (possibly empty). If for each $\epsilon>0$, we can find a solution to the system of inequalities 
\begin{equation}
\label{sol}
|\lambda t-\zeta|<\epsilon\pmod{2\pi}\quad \text{for each $\lambda\in S\backslash S'$}\quad  \text{and} \quad |\lambda t+\zeta|<\epsilon\pmod{2\pi}\quad \text{for each $\lambda\in S'$}
\end{equation}
then there exists a sequence $\{t_k\}$ such that $\cos(\lambda t_k)\rightarrow-\frac{1}{4\alpha}$ for each $\lambda\in S$. As $\alpha=\sum_{\lambda\in S}(E_{\lambda})_{u,u}$, we get
\begin{equation*}
\lim_{k\rightarrow\infty}\sum_{\lambda\in S}\cos(\lambda t_k)(E_\lambda)_{u,u}=\left(\lim_{k\rightarrow\infty}\cos(\lambda t_k)\right)\sum_{\lambda\in S}(E_\lambda)_{u,u}=\bigg(-\frac{1}{4\alpha}\bigg)\alpha=-\frac{1}{4}.
\end{equation*}
Consequently, by Lemma \ref{lem:bipartite}, we have
\begin{equation*}
\begin{split}
\lim_{k\rightarrow\infty}U(t_k)_{u,u} &=(E_{0})_{u,u}+2\lim_{k\rightarrow\infty}\sum_{0\neq \lambda\notin S}\cos(\lambda t_k)(E_{\lambda})_{u,u}+2\lim_{k\rightarrow\infty}\sum_{\lambda\in S}\cos(\lambda t_k)(E_{\lambda})_{u,u} \\
&=(E_{0})_{u,u}+2\lim_{k\rightarrow\infty}\sum_{0\neq\lambda\notin S}\cos(\lambda t_k)(E_{\lambda})_{u,u}-\frac{1}{2}\\
&\leq (E_{0})_{u,u}+2\sum_{0\neq\lambda\notin S}(E_{\lambda})_{u,u}-\frac{1}{2}\leq 0,
\end{split}
\end{equation*}
where the last inequality follows from the fact that $\alpha\geq \frac{1}{4}$. Equivalently, $u$ is not sedentary. Applying Theorem \ref{thm:kronecker} with $S=\{\lambda_1,\ldots,\lambda_r\}$, $\zeta_j=-\zeta$ for all $\lambda_j\in S\backslash S'$ and $\zeta_j=\zeta$ for all $\lambda_j\in S'$, we obtain a solution $t_\epsilon$ to (\ref{sol}) for each $\epsilon>0$ if and only if for each tuple of integers $(\ell_\lambda)$ whose components are not all zero, (\ref{eq:linind}) implies that $D_{S'}\zeta\equiv 0 \pmod{2\pi}$. Since $D_{S'}$ is an integer, the conditions $D_{S'}\zeta\equiv 0 \pmod{2\pi}$ is equivalent to conditions (i) and (ii) in the above theorem.
\end{proof}


\begin{corollary}
\label{cor17}
Let $X$ be a bipartite graph with vertex $u$ such that $0\in \supp_{\be_u}$. If $S$ in Theorem \ref{bip0} is either (i) a linearly independent set over $\Q$, or (ii) consists of integers with equal powers of two in their factorizations, then $u$ is not sedentary.
\end{corollary}

\begin{proof}
If (i) holds, then (\ref{eq:linind}) does not hold for all tuple of integers $(\ell_\lambda)$ whose components are not all zero. Thus, the premise of Theorem \ref{bip0} holds, and so $u$ is not sedentary. Now, suppose (ii) holds and let $\eta$ be the largest power of two that divides each element of $S$. If $\tau=\frac{\pi}{2^\eta}$, then $\cos(\lambda \tau)=-1$ for all $\lambda\in S$, and so Lemma \ref{lem:bipartite} gives us
\begin{equation*}
\begin{split}
U(\tau)_{u,u} &=(E_{0})_{u,u}+2\sum_{0\neq\lambda\notin S}\cos(\lambda \tau)(E_{\lambda})_{u,u}+2\sum_{\lambda\in S}\cos(\lambda \tau)(E_{\lambda})_{u,u} \\
&=(E_{0})_{u,u}+2\sum_{0\neq\lambda\notin S}\cos(\lambda \tau)(E_{\lambda})_{u,u}-2\alpha\\
&\leq \bigg((E_{0})_{u,u}+2\sum_{0\neq\lambda\notin S}(E_{\lambda})_{u,u}\bigg)-2\alpha=(1-2\alpha)-2\alpha\leq 0,
\end{split}
\end{equation*}
where the last inequality uses the fact that $\alpha\geq \frac{1}{4}$. Thus, $u$ is not sedentary.
\end{proof}

\begin{corollary}
\label{cor18}
Let $X$ be a weighted bipartite graph with vertex $u$ such that $0\in\supp_{\be_u}$ and $(E_{0})_{u,u}\leq \frac{1}{2}$. If one of the following conditions hold, then $u$ is not sedentary.
\begin{enumerate}
    \item $|\supp_{\be_u}\cap\mathbb{R}^+|=1$.
    \item $|\supp_{\be_u}\cap\mathbb{R}^+|\geq 2$ and either (i) $\supp_{\be_u}\cap\mathbb{R}^+$ consists of integers with equal powers of two in their factorizations, or (ii) $\supp_{\be_u}\cap\mathbb{R}^+$ is a linearly independent set over $\mathbb{Q}$.
\end{enumerate}
\end{corollary}

\begin{proof}
Suppose 1 holds, and let $\supp_{\be_u}\cap\mathbb{R}^+=\{\lambda\}$. By Lemma \ref{lem:bipartite}, $U(t)_{u,u}=(E_{0})_{u,u}+\cos(\lambda t)(1-(E_{0})_{u,u})$. As $(E_{0})_{u,u}\leq \frac{1}{2}$, we get $U(\pi/\lambda)_{u,u}\leq 0$, and so $u$ is not sedentary. 
Now, suppose 2 holds so that $S=\supp_{\be_u}\cap\mathbb{R}^+$ has at least two elements. If $(E_{0})_{u,u}<\frac{1}{2}$, then Corollary \ref{cor17} implies that $u$ is not sedentary. Meanwhile, if $(E_{0})_{u,u}=\frac{1}{2}$, then conclusion follows from Theorem \ref{thm:sed}(2).
\end{proof}

Let $G(m)$ be the graph obtained by subdividing every edge of the star $K_{1,m}$. Using our results, we show that $G(m)$ is not sedentary.

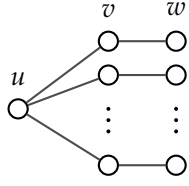
\begin{figure}
\begin{center}
\begin{tikzpicture}[scale=.3,auto=left]
                       \tikzstyle{every node}=[circle, thick, fill=white, scale=0.7]

		        \node[draw,label={above, yshift=0cm: {\Large$u$}}] (1) at (1,0) {};
                \node[draw,label={above, yshift=0cm: {\Large $v$}}] (2) at (5,3) {};
                \node[draw,label={above, yshift=0cm: {\Large $w$}}] (3) at (8,3) {};
                \node[draw] (4) at (5,1.5) {};
                \node[draw] (5) at (8,1.5) {};
                \node[draw] (8) at (5,-2.5) {};
                \node[draw] (9) at (8,-2.5) {};

				\draw [thick, black!70] (1)--(4)--(5);
                \draw [thick, black!70] (1)--(2)--(3);
                \draw [thick, black!70] (1)--(8)--(9);
				
                \foreach \x in {-1,-0.5,0} {\fill (8,\x) circle (2.5pt); }
                \foreach \x in {-1,-0.5,0} {\fill (5,\x) circle (2.5pt); }

		        	\end{tikzpicture}
				
\end{center}	
\caption{\label{fig7} The subdivided star $G(m)$ with vertices $u$ of degree $m$, $v$ of degree two and $w$ of degree one.}
\end{figure}

\begin{corollary}
\label{subdstar}
For all $m\geq 2$, the unweighted subdivided star $G(m)$ is not sedentary.
\end{corollary}

\begin{proof}
Let $m\geq 2$ and consider the graph $G(m)$ in Figure \ref{fig7} with vertex $u$ of degree $m$, vertex $v$ of degree two and leaf $w$. It can be checked that $\spec(G(m))=\left\{\sqrt{m+1},1^{(m-1)},0,-1^{(m-1)},-\sqrt{m+1}\right\}$, where the exponents denote the corresponding eigenvalue multiplicities. Moreover, we have
\begin{equation*}
\supp_{\be_u}=\{0,\pm \sqrt{m+1}\},\quad \supp_{\be_v}=\{\pm \sqrt{m+1},\pm 1\},\quad \text{and}\quad \supp_{\be_w}=\spec(G(m)).
\end{equation*}
Note that an eigenvector $\bv$ for the zero eigenvalue has $u$th entry equal to 1, $w$th entry equal to $-1$ and 0 elsewhere. Since $\supp_{\be_u}$ only has one positive eigenvalue and $(E_0)_{u,u}=\frac{1}{m+1}<\frac{1}{2}$ whenever $m\geq 2$, we get that vertex $u$ is not sedentary by Corollary \ref{cor18}(1). Now, since $0\notin\supp_{\be_v}$, Theorem \ref{thm:bipartite} implies that vertex $v$ is not sedentary.

Finally, we deal with vertex $w$. Note that $\supp_{\be_w}\cap\mathbb{R}^+=\{1,\sqrt{m+1}\}$ and an eigenvector $\bw$ for $\sqrt{m+1}$ has $u$th entry equal to $m$, $w$th entry equal to 1 and $\sqrt{m+1}$ elsewhere. Thus, $(E_{0})_{w,w}=\frac{1}{m+1}$ and $(E_{\sqrt{m+1}})_{w,w}=\frac{1}{2m(m+1)}$, and so $(E_{1})_{w,w}=\frac{1}{2}(1-\frac{1}{m+1}-\frac{1}{m(m+1)})=\frac{m-1}{2m}$. If $m+1$ is not a perfect square, then $\supp_{\be_w}\cap\mathbb{R}^+$ is a linearly independent set over $\mathbb{Q}$. Thus, $w$ is not sedentary by Corollary \ref{cor18}(2ii). On the other hand, if $\sqrt{m+1}$ is an odd integer, then $\supp_{\be_w}\cap\mathbb{R}^+$ consists of integers with equal powers of two in their factorizations. Thus, $w$ is not sedentary by Corollary \ref{cor18}(2i). Lastly, suppose $\sqrt{m+1}$ is an even integer. Then 
\begin{equation*}
\begin{split}
U(\pi)_{w,w}&=(E_{0})_{w,w}+2\cos(\pi\sqrt{m+1}) (E_{\sqrt{m+1}})_{w,w}+2\cos(\pi) (E_{1})_{w,w}\\
&=(E_{0})_{w,w}+2(E_{\sqrt{m+1}})_{w,w}-2(E_{1})_{w,w}=\frac{2-m}{m}\leq 0.
\end{split}
\end{equation*}
Since $U(t)_{w,w}$ is a continuous real-valued function by Lemma \ref{lem:bipartite} and $U(\pi)_{u,u}\leq 0$, we get that $w$ is not sedentary. Combining these observations, we get that no vertex of $G(m)$ is sedentary.
\end{proof}

We will deal more with the subdivision operation in Section \ref{Sec:subd}.

\section{Trees}\label{Sec:tree}

In \cite[Lemma 3.1]{godsil1985inverses}, Godsil showed that a bipartite graph with a unique perfect matching is nonsingular. The proof of this result can be slightly modified to hold for weighted bipartite graphs. Combining this fact with Corollary \ref{cor:intbipgraphs} yields the following result.

\begin{theorem}
\label{thm:bipeveneval}
If $X$ is a bipartite graph with a unique perfect matching, then $X$ is not sedentary for any assignment of edge weights.
\end{theorem}

Since a weighted tree $T$ is bipartite and has at most one perfect matching, we arrive at the next result.

\begin{corollary}
\label{matching}
If $T$ is a tree with a perfect matching, then $T$ is not sedentary for any assignment of edge weights.
\end{corollary}

If $T$ is a tree without a perfect matching, then the conclusion in Corollary \ref{matching} need not hold. We demonstrate this by finding a weighting of an odd path that produces a sedentary end vertex.

\begin{proposition}
\label{oddpath}
Let $n$ be odd and $u$ be an end vertex of $P_n$ incident to an edge of weight $1/\alpha$ for some $\alpha\neq 0$. If $|\alpha|>\sqrt{(n-1)/2}$ and all other edges weights are one, then vertex $u$ is sedentary.  
\end{proposition}

\begin{proof}
Let $V(P_n)=\{u_1,\ldots , u_n\}$ and $E(X)=\{\{u_j,u_{j+1}\}:j=1,\ldots,n-1\}$. Suppose $u=u_1$ so that $\{u_1,u_2\}$ has weight $1/\alpha$, where $|\alpha|>\sqrt{(n-1)/2}$, and all other edge weights are 1. As $n$ is odd, $\theta=0$ is a simple eigenvalue of $A(P_n)$ with eigenvector $[\alpha,0,-1,0,1,\ldots,0,(-1)^{\frac{n-1}{2}}]^T$. Thus $(E_{0})_{u,u}=\frac{\alpha^2}{\alpha^2+\frac{n-1}{2}}$, and so $(E_{0})_{u,u}>\frac{1}{2}$ if and only if $|\alpha|>\sqrt{(n-1)/2}$. Applying Theorem \ref{thm:sed}(1) yields the desired result.
\end{proof}

Note that if $n=2p-1$ for prime $p$ or $n=2^m-1$ for any $m\geq 2$, then the unweighted $P_n$ ($\alpha=1$) admits PGST between end vertices \cite{Godsil2012}. Thus, there are infinite families of odd paths that yield either a sedentary end vertex or PGST between end vertices, depending on the assigned weights.

\begin{theorem}
\label{arbG}
Let $X$ be a weighted graph.
\begin{enumerate}
\item Suppose $X$ has exactly two pendent vertices $u$ and $v$ that share a common vertex. Let $1$ and $\alpha$ be the weights of the edges incident to $u$ and $v$, respectively, for some $\alpha\neq 0$.
\begin{enumerate}
    \item If $|\alpha|\neq 1$, then $u$ or $v$ is sedentary.
    \item Suppose $|\alpha|=1$. Then either $u$ and $v$ are sedentary, or they admit pretty good state transfer. Additionally, if $0$ is an eigenvalue of $A$ with multiplicity at least two and there is a vector $\bv\notin\operatorname{span}\{\bw\}$ in the nullspace of $A$, where $\bw=\be_u-\be_v$ if $\alpha=1$ and $\bw=\be_u+\be_v$ otherwise, such that $\bv^T\be_u\neq 0$, then $u$ and $v$ are sedentary. 
\end{enumerate}
\item If $X$ has at least three pendent vertices that share a common vertex, then $X$ contains at least two sedentary vertices for any assignment of edge weights. In particular, if $X$ is unweighted, then all pendent vertices are sedentary.
\end{enumerate}
\end{theorem}

\begin{proof}
Suppose $P\subseteq V(X)$ is a set of $m\geq 2$ pendent vertices of $X$ that share a common vertex $x$. We may view $X$ as a resulting graph after identifying the central vertex of a star $K_{1,m}$ with a vertex of $H$, where $H$ is the subgraph of $X$ induced by $V(X)\backslash P$, see Figure \ref{fig1} for an illustration when $m=3$. In this case, $x\in V(X)\backslash P$. Without loss of generality, let $1$ and $1/\alpha$ be the weights of the edges $\{u,x\}$ and $\{v,x\}$, respectively for some $\alpha\neq 0$. Since $u$ and $v$ share a common neighbor, $\bv_1=\be_u-\alpha\be_v$ is an eigenvector for $A$ associated with the eigenvalue $\theta=0$.

To prove 1, suppose $m=2$. Then $(E_{\theta})_{u,u}=\frac{1}{1+\alpha^2}$ and $(E_{\theta})_{u,u}=\frac{\alpha^2}{1+\alpha^2}$. If $|\alpha|<1$, then $(E_{\theta})_{u,u}>\frac{1}{2}$. But if $|\alpha|>1$, then $(E_{\theta})_{v,v}>\frac{1}{2}$. In either case, Theorem \ref{thm:sed}(1) implies that either $u$ or $v$ is sedentary. This proves 1a. However, if $|\alpha|=1$, then either $\be_u+\be_v$ or $\be_u-\be_v$ is an eigenvector for $A$, and hence must also be an eigenvector for $U(t)$ for all $t$. Thus, $U(t)(\be_u\pm\be_v)=e^{i\theta}(\be_u\pm\be_v)$, from which it follows that $U(t)_{u,u}\pm U(t)_{u,v}=e^{i\theta}$. Hence, $|U(t)_{u,u}|+|U(t)_{u,v}|\geq 1$, and so if PGST does not occur between $u$ and $v$, then $u$ (and $v$) must be sedentary. This proves the first statement in 1b. The assumption in the second statement of 1b implies that $(E_{\theta})_{u,u}>\frac{1}{2}$, and so  $u$ and $v$ are sedentary by Theorem \ref{thm:sed}(1).

To prove 2, let $m\geq 3$. Let $w$ be pendent vertex distinct from $u$ and $v$, and $1/\beta$ be the weight of the edge $\{w,x\}$ for some $\beta\neq 0$. Then $\bv_2=\be_u-\beta\be_w$ is an eigenvector for $A$ associated with the eigenvalue $0$ that is linearly independent with $\bv_1$. If $V=\{\bv_1,\bv_2,\ldots\}$ is a basis for the nullspace of $A$, then applying the Gram-Schmidt process to $V$ yields an orthonormal basis $W=\{\bw_1,\bw_2,\ldots\}$ for the nullspace of $A$, where 
\begin{equation*}
\bw_1=\frac{1}{\sqrt{1+\alpha^2}}\bv_1\quad 
\text{and} \quad
 \bw_2=\frac{1}{\sqrt{\alpha^2+\alpha^2\beta^2+\beta^2}}\left(\frac{\alpha^2}{\sqrt{1+\alpha^2}}\be_u+\frac{\alpha}{\sqrt{1+\alpha^2}}\be_v- \frac{ \beta (1+\alpha^2)}{ \sqrt{1+\alpha^2}} \be_w\right). 
\end{equation*}
Therefore,
\begin{equation*}
\begin{split}
(E_\theta)_{u,u}&\geq (\bw_1\bw_1^T)_{u,u}+(\bw_2\bw_2^T)_{u,u}
=\frac{1}{1+\alpha^2}\left(1+\frac{\alpha^4}{\alpha^2+\alpha^2\beta^2+\beta^2}\right).
\end{split}
\end{equation*}
Similarly, 
\begin{equation*}
(E_\theta)_{v,v}\geq \frac{\alpha^2}{1+\alpha^2}\left(1+\frac{1}{\alpha^2+\alpha^2\beta^2+\beta^2}\right)\quad \text{and} \quad (E_\theta)_{w,w}\geq \frac{\beta^2(1+\alpha^2)}{\alpha^2+\alpha^2\beta^2+\beta^2}=\frac{1}{\frac{\alpha^2}{\beta^2(1+\alpha^2)}+1}.
\end{equation*}
Note that $(E_\theta)_{u,u}>\frac{1}{2}$ if and only if $\beta^2>\alpha^2(\beta^2-1)$; $(E_\theta)_{v,v}>\frac{1}{2}$ if and only if $\alpha^2(\beta^2+1)>\beta^2$; and $(E_\theta)_{w,w}>\frac{1}{2}$ if and only if $\beta^2(\alpha^2+1)>\alpha^2$. We have two cases.
\begin{itemize}
\item First, suppose $\beta^2>\alpha^2(\beta^2-1)$ so that $(E_\theta)_{u,u}>\frac{1}{2}$. If, additionally, $\beta^2<\alpha^2(\beta^2+1)$, then $(E_\theta)_{v,v}>\frac{1}{2}$. On the other hand, if  $\beta^2>\alpha^2(\beta^2+1)$, then $(E_\theta)_{w,w}>\frac{1}{2}$.
\item Second, suppose $\beta^2\leq \alpha^2(\beta^2-1)$ so that $\beta^2<\alpha^2(\beta^2+1)$ and $(E_\theta)_{v,v}> \frac{1}{2}$. Now, observe that 
$$
\beta^2\leq \alpha^2(\beta^2-1)\Longleftrightarrow\frac{\beta^2}{(\beta^2-1)}\leq \alpha^2\Longleftrightarrow 1-\frac{1}{\beta^2}\geq \frac{1}{\alpha^2}\Longleftrightarrow\beta^2(\alpha^2-1)\geq \alpha^2.
$$
Thus, this second case implies that $\beta^2(\alpha^2+1)>\beta^2(\alpha^2-1)\geq \alpha^2$, and so $(E_\theta)_{w,w}>\frac{1}{2}$.
\end{itemize}
In both cases, there are at least two indices $j\in\{u,v,w\}$ such that $(E_\theta)_{j,j}>\frac{1}{2}$. Invoking Theorem \ref{thm:sed}(1) proves 2.
\end{proof}

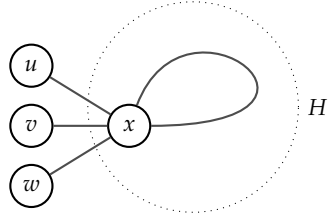
\begin{figure}
\begin{center}
\begin{tikzpicture}[scale=.5,auto=left]
                       \tikzstyle{every node}=[circle, thick, fill=white, scale=0.8]

                \node[draw,minimum size=0.7cm, inner sep=0 pt] (1) at (1.5, 0) {$x$};
		        \node[draw,minimum size=0.7cm, inner sep=0 pt] (3) at (-1.1, 1.6) {$u$};
		        \node[draw,minimum size=0.7cm, inner sep=0 pt] (5) at (-1.1, -1.6) {$w$};
                 \node[draw,minimum size=0.7cm, inner sep=0 pt] (6) at (-1.1, 0) {$v$};
				\node at (6.5,0.5) {$H$};
				\draw[dotted] (3.2,0.5) circle (2.8 cm);
								
				\draw [thick, black!70] (5)--(1)--(3);
                \draw [thick, black!70] (1)--(6);

				\draw[thick, black!70] (1)..controls (3,4) and (8,0)..(1);

				\end{tikzpicture}		
\end{center}
\caption{\label{fig1} A graph $X$ with pendent vertices $u,v$ and $w$ sharing a common neighbor $x$.}
\end{figure}

By Theorem \ref{arbG}, $\{K_{1,n}:n\geq 3\}$ is another infinite family of trees, with no perfect matching, containing sedentary vertices relative to some weighting. In fact, we can say more for this family:

\begin{corollary}
\label{lem20}
The following hold for the weighted star $K_{1,n}$ for all $n\geq 2$.
\begin{enumerate}
\item Suppose $n=2$. If the two edge weights in $K_{1,2}$ are not equal, then one of its leaves is sedentary. Otherwise, perfect state transfer occurs between the leaves of $K_{1,2}$.
\item Suppose $n\geq 3$. Then $K_{1,n}$ yields at least two sedentary vertices for any assignment of edge weights. 
\end{enumerate}   
\end{corollary}

Next, we present the following asymptotic results that imply that vertex sedentariness is a common phenomenon in trees and planar graphs. This complements a result of Godsil on the rarity of PST between vertices \cite[Corollary 10.2]{Godsil2012a}. 

\begin{corollary}
\label{alltrees}
The following hold.
\begin{enumerate}
\item Almost all trees contain at least two sedentary vertices for any assignment of edge weights.
\item For a fixed integer $k\geq 3$, almost all unweighted trees contain $k$ sedentary vertices.
\end{enumerate}
\end{corollary}

\begin{proof}
Adapting the proof of \cite[Theorem 9]{godsil2025quantum}, we get that almost all trees $T$ have $K_{1,3}$ as a limb. That is, almost all trees have the same form as the graph $X$ in Figure \ref{fig1} where $H$ is a tree. Hence, the first statement in Theorem \ref{arbG}(2) yields 1. Next, we apply the same argument but now with the unweighted $K_{1,k}$ as a limb of in lieu of $K_{1,3}$. In this case, applying the second statement in Theorem \ref{arbG}(2) yields 2.
\end{proof}

Corollary \ref{alltrees} strengthens a previous observation about trees, which is that there are infinitely many unweighted trees (with maximum degree three) that contain sedentary vertices \cite{monterde2023new}. Corollary \ref{alltrees} also complements two fundamental results in the literature on trees. The first is that all unweighted trees on $n\geq 4$ vertices do not admit PST \cite[Theorem 9]{cou7}, and the second is that almost all unweighted trees have a property analogous to a sedentary vertex, namely, they contain a sedentary `quantum state' of the form $\frac{1}{\sqrt{2}}(\be_u+\be_v)$ \cite[Theorem 14]{godsil2025quantum}.

We also note that while almost all unweighted trees contain sedentary vertices, there are infinite families of trees where no vertex is sedentary, such as paths admitting PGST between end vertices, as well as subdivided stars as in Corollary \ref{subdstar}.

We further strengthen Corollary \ref{alltrees} with the following result.

\begin{corollary}
\label{planar}
The following hold.
\begin{enumerate}
\item Almost all connected planar graphs contain at least two sedentary vertices for any assignment of edge weights.
\item For a fixed integer $k\geq 3$, almost all connected unweighted planar graphs contain $k$ sedentary vertices.
\end{enumerate}
\end{corollary}

\begin{proof}
Adapting the proof of \cite[Theorem 8]{godsil2025quantum}, we get that almost all connected planar graphs have the same form as the graph $X$ in Figure \ref{fig1}, where $H$ is a connected planar graph. Thus, the first statement in Theorem \ref{arbG}(2) yields 1. The same argument in the proof of Corollary \ref{alltrees}(2) yields 2.
\end{proof}

\section{Bipartite double}\label{Sec:bipdoub}

Let $X$ be a weighted graph. The \textit{bipartite double} of $X$, denoted $K_2\times X$, is the weighted graph with vertex set $\{0,1\} \times V(X)$ and adjacency matrix 
\begin{equation*}
A(K_2)\otimes A(X).
\end{equation*}
We say that vertices $(0,u)$ and $(1,u)$ of $K_2\times X$ are the two copies of vertex $u$ in $X$. Note that $K_2\times X$ is always a bipartite graph. Moreover, $K_2\times X$ is connected if and only if $X$ is not bipartite. It is known that 
\begin{equation}
\label{spectimes}
\spec(K_2\times X)=\{\pm\lambda:\lambda\in\spec(X)\}.
\end{equation}
Thus, if 0 is an eigenvalue of $X$ with multiplicity $m$, then it has multiplicity $2m$ in  $K_2\times X$. In particular, if $\supp_{\be_u}(X)$ and $\supp_{\be_{(0,u)}}(K_2\times X)$ denote the eigenvalue supports of $u$ in $X$ and its copy $(0,u)$ in $K_2\times X$ respectively, then for any vertex $u$ of $X$, we have
\begin{equation}
\label{spectimes1}
\supp_{\be_{(0,u)}}(K_2\times X)=\{\pm\lambda:\lambda\in\supp_{\be_u}(X)\}.
\end{equation}
Moreover, if $E_\lambda$ is an orthogonal projection matrix for $A(X)$, then $(I_2-\frac{1}{2}J)\otimes E_0$ and $\frac{1}{2}J\otimes E_0$ are orthogonal projection matrices associated with the eigenvalues $-\lambda$ and $\lambda$ of $A(K_2\times X)$, respectively.

If $0\in\supp_{\be_u}(X)$, then the following theorem demonstrates that the sedentariness of $u$ in $X$ and its two copies in $K_2\times X$ are equivalent whenever $(E_0)_{u,u}\geq \frac{1}{2}$.

\begin{theorem}
\label{0times}
Let $X$ be a weighted nonbipartite graph with vertex $u$ such that $0\in\supp_{\be_u}(X)$. If $(E_0)_{u,u}>\frac{1}{2}$, then vertex $u$ and its two copies are sedentary in $X$ and $K_2\times X$, respectively. Moreover, if $(E_0)_{u,u}=\frac{1}{2}$, then vertex $u$ is sedentary in $X$ if and only if its two copies are also sedentary in $K_2\times X$.
\end{theorem}

\begin{proof}
Let $0\in\supp_{\be_u}(X)$. Then $0\in\supp_{\be_{(0,u)}}(K_2\times X)$ and the orthogonal projection matrix associated with the eigenvalue $0$ of $A(K_2\times X)$ is given by $F_0=\big((I_2-\frac{1}{2}J)\otimes E_0\big) +\big(\frac{1}{2}J\otimes E_0\big)=I_2\otimes E_0$. Consequently, 
\begin{equation*}
(F_0)_{(0,u),(0,u)}=(F_0)_{(1,u),(1,u)}=(E_0)_{u,u}
\end{equation*}
for any vertex $u$ of $X$. Thus, if $(E_0)_{u,u}>\frac{1}{2}$, then we obtain the desired result by combining the above equation and Theorem \ref{thm:sed}(1). To prove the second statement, let $(E_0)_{u,u}=\frac{1}{2}$. First, suppose that vertex $u$ is sedentary in $X$. By Theorem \ref{thm:sed}(2), there exists a tuple of integers $(\ell_\lambda)_{\lambda\in\supp_{\be_u}(X)\backslash\{0\}}$ such that
\begin{equation*}
\displaystyle\sum_{\lambda\in\supp_{\be_u}(X)\backslash\{0\}}\lambda\ell_\lambda=0\quad  \text{and} \quad \displaystyle\sum_{\lambda\in\supp_{\be_u}(X)\backslash\{0\}}\ell_\lambda\quad \text{is odd}.
\end{equation*}
Now, let $(k_\lambda)_{\lambda\in\supp_{\be_{(0,u)}}(K_2\times X)\backslash\{0\}}$ be a tuple of integers such that $k_\lambda=\ell_\lambda$ for all $\lambda\in\supp_{\be_u}(X)\backslash\{0\}$ and $k_\lambda=0$ otherwise. Then from the above equation, we get that
\begin{equation*}
\displaystyle\sum_{\lambda\in\supp_{\be_{(0,u)}}(K_2\times X)\backslash\{0\}}\lambda k_\lambda=\displaystyle\sum_{\lambda\in\supp_{\be_u}(X)\backslash\{0\}}\lambda\ell_\lambda=0\quad \text{and} \quad \displaystyle\sum_{\lambda\in\supp_{\be_{(0,u)}}(K_2\times X)\backslash\{0\}}k_\lambda=\displaystyle\sum_{\lambda\in\supp_{\be_u}(X)\backslash\{0\}}\ell_\lambda\quad \text{is odd}.
\end{equation*}
Invoking Theorem \ref{thm:sed}(2), the two copies of vertex $u$ in $K_2\times X$ are also sedentary. Next, suppose that vertex $u$ is not sedentary in $X$. By Theorem \ref{thm:sed}(2), for any tuple of integers $(\ell_\lambda)_{\lambda\in\supp_{\be_u}(X)\backslash\{0\}}$, we have
\begin{equation}
\label{lasteq}
\displaystyle\sum_{\lambda\in\supp_{\be_u}(X)\backslash\{0\}}\lambda\ell_\lambda=0\quad  \text{implies that} \quad \displaystyle\sum_{\lambda\in\supp_{\be_u}(X)\backslash\{0\}}\ell_\lambda\quad \text{is even}.
\end{equation}
Now, suppose $\sum_{\mu\in\supp_{\be_{(0,u)}}(K_2\times X)\backslash\{0\}}\mu k_\lambda=0$. Since we may write 
\begin{equation*}
\displaystyle\sum_{\mu\in\supp_{\be_{(0,u)}}(K_2\times X)\backslash\{0\}}\mu k_\lambda=\displaystyle\sum_{\lambda\in\supp_{\be_u}(X)\backslash\{0\}}\lambda(k_\lambda-k_{-\lambda}),
\end{equation*}
if we let $\ell_\lambda=(k_\lambda-k_{-\lambda})$, then (\ref{lasteq}) implies that 
$\sum_{\mu\in\supp_{\be_{(0,u)}}(K_2\times X)\backslash\{0\}} ( k_\lambda - k_{-\lambda})$ is even, which in turn implies that
$\sum_{\mu\in\supp_{\be_{(0,u)}}(K_2\times X)\backslash\{0\}}k_\lambda$ is even. Applying Theorem \ref{thm:sed}(2) once again, we get that the two copies of vertex $u$ in $K_2\times X$ are not sedentary.
\end{proof}

We illustrate Theorem \ref{0times} in the following corollaries. In what follows, $K_{n_1,\ldots,n_m}$ denotes a complete multipartite graph with $m$ parts of sizes $n_1,\ldots,n_m$.

\begin{corollary}
\label{cor27}
Let $m\geq 3$. If $X$ is the bipartite double of $K_{n_1,\ldots,n_m}$, then the following hold. 
\begin{enumerate}
\item If $n_j\geq 3$ for some $j\in\{1,\ldots,m\}$, then the vertices of $K_{n_1,\ldots,n_m}$ in the part of size $n_j$ and their copies in $X$ are sedentary. 
\item If $n_\ell\geq 3$ for all $\ell\in\{1,\ldots,m\}$, then $K_{n_1,\ldots,n_m}$ and $X$ are sedentary graphs.
\item Suppose $n_\ell=2$ for all $\ell\in\{1,\ldots,m\}$. If $m$ is odd, then $K_{n_1,\ldots,n_m}$ and $X$ are sedentary graphs. Otherwise, $K_{n_1,\ldots,n_m}$ and $X$ are not sedentary graphs.
\end{enumerate}
\end{corollary}

\begin{proof}
Since $n\geq 3$, $K_{n_1,\ldots,n_m}$ is not bipartite, and so $K_2\times K_{n_1,\ldots,n_m}$ is connected. To prove 1, suppose $n_j\geq 3$ for some $j\in\{1,\ldots,m\}$. Then the vertices in the part of size $n_j$, which denote by $T$, form a twin set with at least three vertices that are pairwise non-adjacent. Invoking (\ref{twins}), we get that $0\in\supp_{\be_u}(K_{n_1,\ldots,n_m})$ and $(E_0)_{u,u}\geq \frac{2}{3}$ for each $u\in T$. Invoking Theorems \ref{thm:sed}(1) and \ref{0times}, we get that each $u\in T$ is sedentary in $K_{n_1,\ldots,n_m}$ and $K_2\times K_{n_1,\ldots,n_m}$. This proves 1. Note that 2 is immediate from 1. Finally, to prove 3, suppose $n_\ell=2$ for all $\ell\in\{1,\ldots,m\}$. In this case, $K_{n_1,\ldots,n_m}$ is simply the cocktail party graph $K_{2,\ldots,2}$ on $2m$ vertices, which is known to be sedentary if and only if $m$ is odd \cite[Corollary 4]{monterde2023new}. In this graph, every pair of non-adjacent vertices form a twin set of size two, and so $0\in\supp_{\be_u}(K_{n_1,\ldots,n_m})$ and $(E_0)_{u,u}=\frac{1}{2}$ for each vertex $u$ of $K_{2,\ldots,2}$. Combining these facts with Theorem \ref{0times} establishes 3.
\end{proof}

For the case $0\notin\supp_{\be_u}(X)$, we have the following result.

\begin{theorem}
\label{bipdoub}
Let $X$ be a weighted non-bipartite graph with vertex $u$. If $0\notin\supp_{\be_u}(X)$, then the two copies of $u$ in $K_2\times X$ are not sedentary. 
Additionally, if $X$ is nonsingular, then $K_2\times X$ is not sedentary.    
\end{theorem}

\begin{proof}
From (\ref{spectimes1}), $0\notin\supp_{\be_u}(X)$ if and only if $0\notin\supp_{\be_{(0,u)}}(K_2\times X)$. Applying Theorem \ref{thm:bipartite} yields the first statement. The second statement is immediate from the first.
\end{proof}

Let $X$ and $Y$ be unweighted graphs. By $X\vee Y$, we mean the unweighted graph obtained from $X$ and $Y$ by adding all edges with one endpoint in $X$ and the other in $Y$. If we denote the empty graph on $m$ vertices by $O_m$, then a connected \textit{threshold graph} is an unweighted graph that is either of the form:
\begin{equation*}
((((O_{m_1}\vee K_{m_2})\cup O_{m_3})\vee K_{m_4})\ldots)\vee K_{m_{2k}}\quad  \text{or}
\quad ((((K_{m_1}\cup O_{m_2})\vee K_{m_3})\cup O_{m_4})\ldots)\vee K_{m_{2k+1}}
\end{equation*}
where each $m_j$ is an integer such that $m_1\geq 2$ and $m_j\geq 1$ for all $j\geq 2$ \cite{Severini}. The vertex subsets formed by each of $O_{m_j}$ and $K_{m_j}$ are called cells with size $m_j$.

\begin{corollary}
Let $X$ be an unweighted threshold graph with a cell $\mathcal{C}$ of size $m_j\geq 3$. If the subgraph induced by $\mathcal{C}$ is empty, then each vertex in $\mathcal{C}$ is sedentary in $X$ and $K_2\times X$. Otherwise, each vertex in $\mathcal{C}$ is sedentary in $X$ but not in $K_2\times X$.
\end{corollary}

\begin{proof}
If $m_j\geq 3$, then the vertices in $\mathcal{C}$ form a twin set, and hence are sedentary in $X$ by Lemma \ref{monterde}. Now, if the subgraph induced by $\mathcal{C}$ is empty, then because $m_j\geq 3$, each vertex in $\mathcal{C}$ is sedentary in $K_2\times X$ by Lemma \ref{monterde}. Now, suppose the subgraph induced by $\mathcal{C}$ is complete. If $X=((((O_{m_1}\vee K_{m_2})\cup O_{m_3})\vee K_{m_4})\ldots)\vee K_{m_{2k}}$, then by \cite[Lemma 3]{ghorbani2019eigenvalue}, $0\in\spec(X)$ has multiplicity equal to $\sum_{j\ \text{odd}}(m_j-1)$. 
This implies that $E_\lambda$ is a direct sum of $I_{m_j}-\frac{1}{m_j}J_{m_j}$ for all odd $j$ and a zero matrix of order $n-\sum_{j\ \text{odd}}m_j$. Consequently, 
$0\notin\supp_{\be_u}(X)$ for all $u\in \mathcal{C}$. A similar argument yields the same conclusion when $X=((((K_{m_1}\cup O_{m_2})\vee K_{m_3})\cup O_{m_4})\ldots)\vee K_{m_{2k+1}}$. Applying Theorem \ref{bipdoub} yields the result.
\end{proof}

If $X=K_n$, then Theorem \ref{bipdoub} recovers the fact that the bipartite double of $K_n$ for all $n\geq 3$ is not sedentary \cite[Corollary 35]{monterde2023new}. For the complete graph minus an edge $K_n\backslash e$, we have the next result.

\begin{corollary}
Let $n\geq 5$ and suppose $u$ and $v$ are vertices in $K_n\backslash e$ such that $u$ is an endpoint of $e$ but $v$ is not. Then $u$ is not sedentary in $K_n\backslash e$ and $K_2\times (K_n\backslash e)$, while $v$ is sedentary in $K_n\backslash e$ but not in $K_2\times (K_n\backslash e)$.
\end{corollary}

\begin{proof}
Note that $0\in\supp_{\be_u}(K_n\backslash e)$ with $(E_0)_{u,u}=\frac{1}{2}$. By \cite[Example 3]{Monterde2022}, $u$ is involved in PGST in $K_n\backslash e$, and so it is not sedentary. Applying the second statement in Theorem \ref{0times} yields the conclusion for vertex $u$. Now, since $n\geq 5$, vertex $v$ belongs to a twin set of size at least three, and so $v$ is sedentary by Lemma \ref{monterde}. Since $0\notin\supp_{\be_v}(K_n\backslash e)$, Theorem \ref{bipdoub} completes the proof.
\end{proof}

Since $C_{2n}=K_2\times C_n$, Theorem \ref{bipdoub} also gives us the following corollary.

\begin{corollary}
\label{cnodd}
Let $n$ be odd and $X$ be a weighted $C_n$. If $X$ is nonsingular, then the weighted $C_{2n}$ obtained from taking the bipartite double of $X$ is not sedentary.
\end{corollary}

\section{Subdivision}\label{Sec:subd}

Let $X$ be a weighted graph. A \textit{subdivision} $S(X)$ of $X$ is the weighted graph obtained from $X$ by replacing each edge $\{u,v\}$ of $X$ by the edges $\{u,w\}$ and $\{w,v\}$ whose weights are equal to that of $\{u,v\}$. Note that $S(X)$ is always bipartite and has exactly $|V(X)|+|E(X)|$ number of vertices.

Adapting the proofs of Theorems 4 and 5 in \cite{gervacio1990bipartite} yields the next two results.

\begin{lemma}
\label{unicyc}
Let $X$ be a connected graph. If $S(X)$ is nonsingular, then $X$ is a unicyclic weighted graph.
\end{lemma}

\begin{proof}
Let $X$ be connected and $S(X)$ be nonsingular. To show that $X$ is unicyclic, we show that $X$ is neither a tree nor a graph with at least two cycles. If $X$ is a tree, then $|E(X)|=|V(X)|-1$, and so $|V(S(X))|=|V(X)|+|E(X)|=2|V(X)|-1$. Hence $|V(S(X))|$ is always odd, and so $S(X)$ has no perfect matching. By Corollary \ref{matching}, 0 is an eigenvalue of $A(S(X))$. Now, if $X$ has at least two cycles, then $|E(X)|>|V(X)|$, and so $|V(S(X))|=|V(X)|+|E(X)|>2|E(X)|$. Note that the $|E(X)|$ new vertices added to $X$ form an independent set in $S(X)$. Since $|E(X)|>\frac{|V(S(X))|}{2}$, 0 is an eigenvalue of $A(S(X))$ by interlacing. In both cases, we see that $S(X)$ is singular, a contradiction. Thus, $X$ must be unicyclic.
\end{proof}

In the next proof, $X\backslash V'$ denotes the subgraph of $X$ induced by $V(X)\backslash V'$ for some $V'\subseteq V(X)$.

\begin{theorem}
\label{subd}
Let $X$ be a connected weighted graph. Then $S(X)$ is nonsingular if and only if $X$ is a unicyclic graph with a weighted cycle $C_p$ as a subgraph such that $S(C_p)$ is nonsingular.
\end{theorem}

\begin{proof}
By virtue of Lemma \ref{unicyc}, it suffices to show that for a connected unicyclic graph $X$ with a weighted cycle $C_p$ as a subgraph, $S(X)$ is nonsingular if and only if $S(C_p)$ is nonsingular. Let $X$ be such a graph. If $X=C_p$, then there is nothing to prove. So, suppose $X$ has a cycle $C_p$ as a subgraph and a pendent vertex. Let $u$ be a pendent vertex in $S(X)$ adjacent to vertex $v$ with degree two and let $w$ be the other neighbor of $v$. By definition of $S(X)$, the edges $\{u,v\}$ and $\{v,w\}$ have equal weights, say $\omega$. If we index the first, second and third rows of $A(S(X))$ by vertices $u$, $v$ and $w$ respectively, then we may write
\begin{equation*}
A(S(X))=\left[\begin{array}{c|c|ccc}
     0 & \omega &0&\zero&\\\hline
     \omega& 0&\omega&\zero& \\\hline
     0&\omega & &&\\
     &&&&\\
     \zero&\zero&&A(S(X)\backslash\{u,v\})&\\
     &&&&
\end{array}\right].
\end{equation*}
Suppose we subtract the first row of $A(S(X))$ from its third row. This elementary row operation yields a matrix in block triangular form with one block equal to $A(S(X)\backslash\{u,v\})$. Thus, $\det A(S(X))\neq 0$ if and only if $\det A(S(X)\backslash\{u,v\})\neq 0$. Applying this argument exactly $|X\backslash V(C_p)|$ times starting with $A(S(X))$ successively removes pairs of edges in each component of $S(X)\backslash S(C_p)$, where each pair of edges comes from a subdivided edge in $X$. In particular, the last step in this inductive process allows us to deduce that $\det A(S(X))\neq 0$ if and only if $\det A(S(C_p))\neq 0$.
\end{proof}

For unweighted graphs, $S(C_p)$ is nonsingular if and only if $p$ is odd. Applying Theorem \ref{subd}, we get a result first proved in \cite[Theorem 7]{gervacio1990bipartite}.

\begin{corollary}
\label{subd1}
Let $X$ be a connected unweighted graph. Then $S(X)$ is nonsingular if and only if $X$ is a nonbipartite unicyclic graph.
\end{corollary}

We now combine Corollary \ref{cor:intbipgraphs} with Theorem \ref{subd} and Corollary \ref{subd1}.

\begin{theorem}
\label{subdsed}
Let $X$ be a connected weighted graph. If $X$ is a unicyclic graph with a weighted cycle $C_p$ as a subgraph such that $S(C_p)$ is nonsingular,
then $S(X)$ is not sedentary. In particular, if $X$ is an unweighted nonbipartite unicyclic graph, then $S(X)$ is not sedentary.
\end{theorem}

\begin{corollary}
\label{c6}
Let $X$ be a connected unicyclic graph with a cycle $C_3$ as a subgraph. Then $S(X)$ is a planar graph that is not sedentary for any assignment of edge weights on $X$.
\end{corollary}

\begin{proof}
Note that any assignment of edge weights on $X$ yields a weighted $C_3$ as a subgraph of $X$, say with edge weights $\alpha,\beta,\gamma\neq 0$. Since $S(C_3)$ is nonsingular for all $\alpha,\beta,\gamma\neq 0$, we get that $S(X)$ is not sedentary by Theorem \ref{subdsed} for any assignment of edge weights on $X$.
\end{proof}

Note that if $T$ is a tree, then adding an edge between any two vertices sharing a common neighbor results in a unicyclic graph with $C_3$ as a subgraph. One may then apply Corollary \ref{c6} to construct infinitely many planar subdivision graphs $S(X)$ that are not sedentary for any assignment of edge weights on $X$.

\begin{remark}
Let $X$ be an unweighted graph. The following hold.
\begin{enumerate}
\item If $S(X)$ is nonsingular, then it is a bipartite graph that need not have a unique perfect matching, see the graph $X'$ in Example \ref{ex1}(1).
\item There are connected bipartite unicyclic graphs that are singular, see the graph $X$ in Example \ref{ex1}(1). However, such graphs are not subdivision graphs by Theorem \ref{subd}.
\end{enumerate}
\end{remark}

\section{Corona product}\label{Sec:corona}

Let $X$ be a weighted graph with vertices $v_1,\ldots,v_n$ and $Y$ be a weighted $k$-regular graph on $m$ vertices. The \textit{corona product} $X\circ Y$ of $X$ and $Y$ is formed by taking the disjoint union of $X$ and $n$ copies of $Y$, and then adding an edge of weight one from each vertex of the $j$th copy of $Y$ to vertex $v_j$ of $X$. This graph product has adjacency matrix given by
\begin{equation*}
A(X\circ Y)=\begin{bmatrix} A(X) & I_n\otimes \mathbf{1}^T \\ I_n\otimes \mathbf{1} & I_n\otimes A(Y)   \end{bmatrix}.
\end{equation*}
Let $u$ be a vertex of $X$. If $\supp_{\be_u}(X)$ and $\supp_{\be_u}(X\circ Y)$ denote the eigenvalue supports of $u$ in $X$ and $X\circ Y$ respectively, then by \cite[Proposition 3.2]{ack1}, we obtain, 
\begin{equation*}
\supp_{\be_u}(X\circ Y)=\{\lambda^\pm:\lambda\in\supp_{\be_u}(X)\},
\end{equation*}
where
\begin{equation}
\label{mu}
\lambda^\pm=\frac{1}{2}\bigg(\lambda+k\pm\sqrt{(\lambda-k)^2+4m}\bigg)
\end{equation}
and
\begin{equation}
\label{uu}
(E_{\lambda^\pm})_{u,u}=\frac{(\lambda^\pm-k)^2}{(\lambda^\pm-k)^2+m}(E_{\lambda})_{u,u}.
\end{equation}

The following result is useful whenever $X$ is a non-bipartite graph.

\begin{theorem}
\label{corprod}
Suppose $X$ is a weighted graph and $Y$ is a weighted $k$-regular graph on $m$ vertices. If $(E_\lambda)_{u,u}>\min\big\{\frac{(\lambda^\pm-k)^2+m}{2(\lambda^\pm-k)^2}\big\}$ for some $\lambda\in\spec(A(X))$, then $u$ is sedentary in $X$ and $X\circ Y$.
\end{theorem}

\begin{proof}
If $(E_\lambda)_{u,u}>\min\big\{\frac{(\lambda^{\pm}-k)^2+m}{2(\lambda^{\pm}-k)^2}\big\}$, then by~(\ref{uu}), we get that $(E_{\lambda^\pm})_{u,u}>\frac{1}{2}$. Hence, $(E_\lambda)_{u,u}$ and $(E_{\lambda^\pm})_{u,u}$ are both greater than $1/2$. Applying Theorem \ref{thm:sed} yields the desired result.
\end{proof}

\begin{corollary}
\label{tw}
Let $X$ be an unweighted graph with a twin set $W$ whose induced subgraph is a clique, and let $Y$ be a weighted $k$-regular graph on $m\geq 1$ vertices. If $|W|>2+\frac{2m}{(k-\lambda^-)^2-m}$ where $\lambda^-$ is derived from $\lambda=-1$, then $X\circ Y$ is sedentary at every vertex of $W$.    
\end{corollary}

\begin{proof}
Let $u\in W$. By assumption, $X$ is unweighted and the induced subgraph of $W$ is a clique. Thus, $\lambda=-1\in\spec(A(X))$, and so
$\lambda^-=\frac{1}{2}\big(k-1-\sqrt{(k+1)^2+4m}\big) \in\spec(A(X\circ Y))$. From (\ref{twins}), $(E_{\lambda})_{u,u}\geq \frac{|W|-1}{|W|}$. Since $\frac{|W|-1}{|W|}>\frac{(\lambda^--k)^2+m}{2(\lambda^--k)^2}$ if and only if $|W|>2+\frac{2m}{(k-\lambda^-)^2-m}$, Theorem \ref{corprod} yields the desired result.
\end{proof}
 
Taking $Y=O_m$ in Corollary \ref{tw} yields the next result.

\begin{corollary}
\label{|W|}
Let $X$ be an unweighted graph with a twin set $W$ whose induced subgraph is a clique and $Y=O_m$ for some $m\geq 1$. If $|W|>2+\frac{4m}{\sqrt{4m+1}+1}$, then $X\circ O_m$ is sedentary at every vertex of $W$.   
\end{corollary}

\begin{corollary}
\label{KO}
The following hold in the unweighted $K_n\circ O_m$.
\begin{enumerate}
\item Let $m=1$. If $n\geq 4$, then $K_n\circ O_1$ is sedentary at every vertex of $K_n$.
\item Let $m\geq 2$. If $n>2+\frac{4m}{\sqrt{4m+1}+1}$, then $K_n\circ O_m$ is sedentary at every vertex in $K_n$. 
Moreover, every leaf vertex in $K_n\circ O_2$ is not sedentary for all $n\geq 1$, while every leaf vertex in $K_n\circ O_m$ for $m\geq 3$ is sedentary for all $n\geq 1$.
\end{enumerate}
\end{corollary}

\begin{figure}[h!]
	\begin{center}
		\begin{tikzpicture}
		\tikzset{enclosed/.style={draw, circle, inner sep=0pt, minimum size=.2cm}}	   
        \node[enclosed] (a) at (13,1) {};
        \node[enclosed] (b) at (13,0) {};
        \node[enclosed] (c) at (14,1) {};
        \node[enclosed] (d) at (14,0) {};
        \node[enclosed, fill=black] (e) at (12.3,1.1) {};
         \node[enclosed, fill=black] (ee) at (12.5,1.5) {};
        \node[enclosed, fill=black] (f) at (12.9,1.7) {};
        \node[enclosed, fill=black] (g) at (12.3,-0.1) {};
        \node[enclosed, fill=black] (ff) at (12.5,-0.5) {};
        \node[enclosed, fill=black] (h) at (12.9,-0.7) {};
        \node[enclosed, fill=black] (e1) at (14.7,1.1) {};
        \node[enclosed, fill=black] (e1e1) at (14.5,1.5) {};
        \node[enclosed, fill=black] (f1) at (14.1,1.7) {};
        \node[enclosed, fill=black] (g1) at (14.7,-0.1) {};
        \node[enclosed, fill=black] (f1f1) at (14.5,-0.5) {};
        \node[enclosed, fill=black] (h1) at (14.1,-0.7) {};
        
		\draw (b) -- (a) -- (c) -- (d) -- (b);
        \draw (f) -- (a) -- (e);
        \draw (g) -- (b) -- (h);
        \draw (f1) -- (c) -- (e1);
        \draw (g1) -- (d) -- (h1);
\draw (a) -- (ee);
\draw (b) -- (ff);
\draw (c) -- (e1e1);
\draw (d) -- (f1f1);
        \node[enclosed, fill=black] (1) at (6.65,0.95) {};
        \node[enclosed, fill=black] (2) at (7,0) {};
        \node[enclosed, fill=black] (3) at (8,0) {};
        \node[enclosed, fill=black] (4) at (8.35,0.95) {};
        \node[enclosed,fill=black] (5) at (7.5,1.6) {};
        \node[enclosed] (11) at (6,0.7) {};
        \node[enclosed] (12) at (6.1,1.3) {};
        \node[enclosed] (21) at (6.35,-0.35) {};
        \node[enclosed] (22) at (6.95,-0.7) {};
        \node[enclosed] (41) at (9,0.7) {};
        \node[enclosed] (42) at (8.9,1.3) {};
        \node[enclosed] (51) at (7.2,2.2) {};
        \node[enclosed] (52) at (7.8,2.2) {};
        \node[enclosed] (31) at (8.65,-0.35) {};
        \node[enclosed] (32) at (8.05,-0.7) {};
        
		\draw (1) -- (2) -- (3) -- (4) -- (5) -- (1) -- (3) --(5) -- (2) -- (4) -- (1);
        \draw (11) -- (1) -- (12);
        \draw (21) -- (2) -- (22);
        \draw (51) -- (5) -- (52);
        \draw (41) -- (4) -- (42);
        \draw (31) -- (3) -- (32);

		\end{tikzpicture}
	\end{center}
	\caption{Unweighted graphs with sedentary vertices marked black: $K_5\circ O_2$ (left) and $C_4\circ O_3$ (right)}\label{yo}
\end{figure}

\begin{proof}
Taking $X=K_n$ (so that $|W|=n$) in Corollary \ref{|W|} yields the conclusion in 1-2 about the vertices in $K_n$ in $K_n\circ O_m$. Now, if $m\geq 2$, then the $m$ leaves sharing a common neighbor are twins. By the first statement of Lemma \ref{monterde}, these vertices are either sedentary, or they admit PGST. If $m=2$, then they admit PGST in $K_n\circ O_m$ by \cite[Theorem 5.4]{ack1}. Hence, they are not sedentary. Now, if $m=3$, then the second statement in Lemma \ref{monterde} yields the desired conclusion.
\end{proof}

For the next result, we note that $X\circ Y$ is bipartite if and only if $X$ is and $Y=O_m$.

\begin{theorem}
\label{circY}
Let $m\geq 1$ and $X$ be a bipartite graph. If $u$ is a vertex of $X$ in $X\circ O_m$, and $v$ is a leaf vertex in $X\circ O_m$, then the following hold.
\begin{enumerate}
\item Vertex $u$ is not sedentary in $X\circ O_m$ for any assignment of edge weights on $X$. 
\item If $m=1$, then vertex $v$ is not sedentary for any assignment of edge weights on $X$.
\item If $m=2$, then vertex $v$ is not sedentary whenever (i) $\supp_{\be_u}\cap\mathbb{R}^+$ consists of integers with equal powers of two in their factorizations, or (ii) $\supp_{\be_u}\cap\mathbb{R}^+$ is a linearly independent set over $\mathbb{Q}$.
\item Let $m\geq 3$. Then vertex $v$ is sedentary for any assignment of edge weights on $X$. 
Moreover, at least two leaf vertices attached to $u$ is sedentary in $X\circ O_m$ for any assignment of edge weights on $X\circ O_m$.
\end{enumerate}
\end{theorem}

\begin{proof}
Because $O_m$ is $0$-regular, (\ref{mu}) yields
$\lambda^\pm=\frac{1}{2}(\lambda\pm\sqrt{\lambda^2+4m})$ for each $\lambda\in\supp_{\be_u}(X)$. So $\lambda^\pm=0$ if and only if $m=0$, a contradiction. Hence, $0\notin\supp_{\be_u}(X\circ O_m)$ for any assignment of weights on $X$. Invoking Theorem \ref{thm:bipartite} yields 1. If $m=1$, then \cite[Proposition 3.2]{ack1} gives us $\supp_{\be_v}(X\circ O_1)=\supp_{\be_u}(X\circ O_1)$. Thus, $0\notin \supp_{\be_v}(X\circ O_1)$, and so $v$ is not sedentary by Theorem \ref{thm:bipartite}. Now, let $m\geq 2$ so that the vertices in $O_m$ form a twin set in $X\circ O_m$. If $m=2$, then $\supp_{\be_v}(X\circ O_2)=\supp_{\be_u}(X\circ O_2)\cup\{0\}$ and $(E_0)_{v,v}=\frac{1}{2}$. As $X\circ O_2$ is bipartite, invoking Corollary \ref{cor18}(2) yields 3. Lastly, the first statement in 4 follows from Lemma \ref{monterde}, while the second statement follows from Theorem \ref{arbG}(2).
\end{proof}

\begin{example}
The following hold for the graphs in Figure \ref{yo}.
\begin{enumerate}
\item In the unweighted $K_5\circ O_2$, a vertex of $K_5$ is sedentary but a leaf vertex is not by Corollary \ref{KO}(2). 
\item In $C_4\circ O_3$, each vertex of $C_4$ is not sedentary by Theorem \ref{circY}(1), while each leaf vertex is sedentary by  Theorem \ref{circY}(2). This holds for any assignment of edge weights on $C_4$ in $C_4\circ O_3$.
\end{enumerate}
\end{example}

\begin{corollary}
\label{directcirc}
Suppose $X$ is a weighted nonbipartite graph and $Y$ is a weighted $k$-regular graph on $m$ vertices. If $m\neq\lambda k $ for all $\lambda\in\spec(A(X))$, then $K_2\times (X\circ Y)$ is not sedentary.
\end{corollary}

\begin{proof}
If $X$ is nonbipartite, then so is $X\circ Y$. By (\ref{mu}), $A(X\circ Y)$ has a zero eigenvalue if and only if $m=\lambda k$. Invoking Theorem \ref{bipdoub} yields the desired result.
\end{proof}

We close this section by noting that there are many graphs for which $m\neq\lambda k $ for all $\lambda\in\spec(A(X))$. Indeed, if the characteristic polynomial of $A(X)$ has integer coefficients and $k\neq 0$ is an integer that does not divide $m$, then $m\neq\lambda k $ for all $\lambda\in\spec(A(X))$. Otherwise, $\lambda=\frac{m}{k}\in\spec(A(X))$ is a noninteger rational, a contradiction to the fact that each $\lambda\in\spec(A(X))$ is either an integer or a quadratic integer.

\section{Bipartite graphs with few distinct eigenvalues}\label{Sec:bip2}

In this section, we prove three results. First, we characterize sedentariness in bipartite graphs with $k\leq 4$ distinct eigenvalues. Second, we construct an infinite family of bipartite graphs with five distinct eigenvalues with all but one vertex that is sedentary. Combining this second result with Corollary \ref{subdstar}, we obtain two infinite families of graphs with five distinct eigenvalues, in which we can determine the sedentariness of every vertex. And finally, we characterize sedentariness on trees with diameter at most three.

We begin with the observation that $K_2$ is the only bipartite graph with two distinct eigenvalues and $K_2$ is known to not be sedentary. Next, connected unweighted bipartite graphs with three distinct eigenvalues are precisely the complete bipartite graphs $K_{m,n}$ \cite{van1998nonregular}. We characterize sedentary vertices in $K_{m,n}$ below.

\begin{proposition}
\label{kmn}
Let $u$ be a vertex in an unweighted $K_{m,n}$ that belongs to a part of size $m$. Then vertex $u$ is sedentary if and only if $m\geq 3$.
\end{proposition}

\begin{proof}
If $m=1$, then $U(\frac{\pi}{2\sqrt{n}})_{u,u}=0$ by \cite[Example 11]{monterde2023new}. Meanwhile, if $m=2$, then vertex $u$ admits PST \cite[Theorem 12(1i)]{Monterde2022}. Combining this with Corollary \ref{cor27}(1) yields the desired result.
\end{proof}

A \textit{multiplicative design} is an incidence structure $\mathcal{S}=(V,\mathcal{B},I)$, with $|V|=v$, $|\mathcal{B}|=b$ such that the
incidence matrix $N$ of $\mathcal{S}$ (with rows and columns indexed by points and blocks, respectively) satisfies 
\begin{equation*}
NN^T=D+\bv\bv^T
\end{equation*}
for some diagonal matrix $D$, with positive diagonal entries, and some vector $\bv$ with real positive entries. If the matrix $D$ is a scalar matrix $dI$, then such a multiplicative design is said to be \textit{uniform}. Bipartite graphs with four distinct eigenvalues are precisely the incidence graphs of uniform multiplicative designs \cite{van2004combinatorial}. Combining this fact with Corollary \ref{evenevals} yields the following result.

\begin{corollary}
The unweighted incidence graphs of uniform multiplicative designs are not sedentary.
\end{corollary}

For an infinite family of bipartite graphs constructed from uniform multiplicative designs, see \cite{ryser1972symmetric}.

Next, we utilize the construction in \cite[Section 2]{van2014graphs} to build bipartite graphs with five distinct eigenvalues that contain sedentary vertices. Let $H$ be a graph on $n$ vertices which is the disjoint union of $f>1$ mutually non-isomorphic complete bipartite graphs, each with $e$ edges (note that $e$ is not prime, otherwise $f=1$, a contradiction). Let $G$ be the graph on $n+1$ vertices obtained from $H$ by adding a new vertex labeled $v$ that is joined to all vertices of one color class of each of the components of $H$. Then
\begin{equation}
\label{eq1}
\spec(G)=\left\{\lambda,\sqrt{e}^{(f-1)},0^{(n-2f+1)},-\sqrt{e}^{(f-1)}, -\lambda\right\},
\end{equation}
where $\lambda$ is the spectral radius of $A$. Since $\operatorname{tr}A^2=2|E(G)|$, we get $\lambda^2+e(f-1)=ef+\operatorname{deg}v$. Hence,
\begin{equation*}
\lambda=\sqrt{e+\operatorname{deg}v},
\end{equation*}
and so $G$ has exactly five distinct eigenvalues.

\begin{theorem}
\label{21}
Let $G$ be the unweighted bipartite graph constructed above with parts $B_1\cup\{v\}$ and $ B_2$, and let $H'$ be a component of $H$.
\begin{enumerate}
\item If $u$ belongs to a part of $H'$ of size at least three, then $u$ is sedentary. In particular, if no component of $H$ is isomorphic to $K_{1,e}$ and $K_{2,e/2}$ (say, when $e$ is even), then every vertex in $G$, other than $v$, is sedentary.
\item Let $e\geq 4$ be even and $H'=K_{2,\frac{e}{2}}$. If $u$ belongs to a part of $H'$ of size two, then the following hold.
\begin{enumerate}
\item If $u\in B_1$, then $u$ is sedentary. 
\item If $u\in B_2$, then $u$ is not sedentary if and only if $a\lambda+b\sqrt{e}=0$ implies that $a+b$ is even for all integers $a,b$. In particular, $u$ is not sedentary if $\lambda$ and $\sqrt{e}$ are linearly independent over $\mathbb{Q}$.
\end{enumerate}
\item We have $\supp_{\be_v}=\{\pm\lambda, 0\}$, and vertex $v$ is sedentary in $G$ if and only if 
$\sum_{j=1}^f\frac{1}{n_j}<1$, where $n_1,\ldots,n_f$ are the sizes of the parts of the components of $H$ that belong to $B_1$.
\item Let $e\geq 4$ and $H'=K_{1,e}$. If $u$ is a vertex in $H'$ of degree $e$, then the following hold.
\begin{enumerate}
\item If $u\in B_1$, then $u$ is not sedentary whenever (i) $\lambda$ and $\sqrt{e}$ are integers with $\nu_2(\lambda)=\nu_2(\sqrt{e})$, or (ii) $\lambda$ and $\sqrt{e}$ are linearly independent over $\mathbb{Q}$.
\item If $u\in B_2$, then $u$ is not sedentary.
\end{enumerate}
\end{enumerate}
\end{theorem}
\begin{proof}
Note that $G$ preserves twin vertices in $H$. Therefore, the first statement in 1 is immediate from Lemma \ref{monterde}. In particular, if no component of $H$ is isomorphic to $K_{1,e}$ and $K_{2,e/2}$, then each vertex of $G$ other than $v$ belongs to a twin set of size at least three. From this, the conclusion is immediate. 

To prove 2, let $H'=K_{2,\frac{e}{2}}$ for some even $e\geq 4$, and suppose  $u$ belongs to a part of $H'$ of size two. That is, $u$ belongs to a twin set of size two, and so by (\ref{twins}), we obtain $(E_0)_{u,u}=\frac{1}{2}+F_{u,u}$,
where $F_{u,u}\geq 0$. Note that $n-2f$ zero eigenvalues of $A(G)$ arise from the fact that $G$ preserves all twin vertices in $H$, each of which has associated eigenvector of the form $\be_a-\be_b$ where $a$ and $b$ belong to the same part of a component in $H$. Since 0 has multiplicity $n-2f+1$ by (\ref{eq1}), there is an eigenvector for 0 that is orthogonal to each $\be_a-\be_b$. This eigenvector will determine whether $F_{u,u}>0$ or $F_{u,u}=0$. For each $j\in\{1,\ldots, f\}$, let $H_j=K_{n_j,m_j}$ be the $j$th component of $H$ and assume that $v$ is adjacent to each vertex in the part of size $m_j$ for each $j$. That is, $n_1,\ldots,n_f$ are the sizes of the parts of the components of $H$ that belong to $B_1$. Now, suppose we index the rows of $A(G)$ by the $n_1$ vertices of $H_1$, followed by the $n_2$ vertices of $H_2$, and so on until the $n_f$ vertices of $H_f$, followed by $v$, followed by the $m_1$ vertices of $H_1$, followed by the $m_2$ vertices of $H_2$, and so on until the $m_f$ vertices of $H_f$. Set
\begin{equation*}
N=\left[\begin{array}{ccccc}
     J_{n_1,m_1}&O&\cdots&O\\
     O&J_{n_2,m_2}&\cdots& O\\
    \vdots&\vdots&\ddots&\vdots\\
     O&O&\cdots&J_{n_f,m_f}\\
    \mathbf{1}^T& \mathbf{1}^T&\cdots&\mathbf{1}^T 
\end{array}\right]\quad \text{and} \quad \mathbf{v}=\left[\begin{array}{c}
     \frac{1}{n_1}\mathbf{1}_{n_1}\\\frac{1}{n_2}\mathbf{1}_{n_2}\\    \vdots\\ \frac{1}{n_f}\mathbf{1}_{n_f} \\
    -1
\end{array}\right].
\end{equation*}
\begin{figure}[h!]
	\begin{center}
		\begin{tikzpicture}
		\tikzset{enclosed/.style={draw, circle, inner sep=0pt, minimum size=.2cm}}	   
	    \node[enclosed, fill=cyan] (a) at (0,0.35) {};
        \node[enclosed, fill=cyan] (b) at (0,-0.35) {};
        \node[enclosed, fill=cyan] (c) at (0.5,0.8) {};
        \node[enclosed, fill=cyan] (d) at (0.5,-0.8) {};
        \node[enclosed] (v) at (1.2,0) {};
        \node[enclosed, fill=black, label={above, yshift=0cm: $v$}] (u) at (2.1,0) {};
        \node[enclosed] (e) at (3,0.7) {};
        \node[enclosed] (f) at (3,-0.7) {};
        \node[enclosed, fill=magenta] (g) at (4,0.7) {};
        \node[enclosed, fill=magenta] (h) at (4,-0.7) {};
		\draw (a) -- (v);
		\draw (b) -- (v);
		\draw (c) -- (v);
		\draw (d) -- (v);
        \draw (e) -- (g);
        \draw (e) -- (h);
        \draw (f) -- (g);
        \draw (f) -- (h);
        \draw (u) -- (v);
        \draw (u) -- (e);
        \draw (u) -- (f);
		\end{tikzpicture}
	\end{center}
	\caption{The graph $G$ with vertex $v$ and a color class consisting of nonsedentary vertices marked white}\label{yay}
\end{figure}
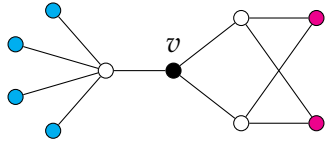
Then we may write $A(G)=\left[\begin{array}{cc}O&N\\ N^T&O \end{array}\right]$ and we have $A(G)\left[\begin{array}{cc}\mathbf{v}\\ \mathbf{0}\end{array}\right]=0$. Thus, $\left[\begin{array}{cc}\mathbf{v}\\ \mathbf{0}\end{array}\right]$ is an eigenvector for 0 that is orthogonal to each $\be_a-\be_b$. If $u\in B_1$, then $\be_u^T\left[\begin{array}{cc}\mathbf{v}\\ \mathbf{0}\end{array}\right]\neq 0$, and so $F_{u,u}>0$. Therefore, $(E_0)_{(u,u)}>\frac{1}{2}$ and so applying Theorem \ref{thm:sed}(1) yields 2a. Meanwhile, if $u\in B_2$, then $\be_u^T\left[\begin{array}{cc}\mathbf{v}\\ \mathbf{0}\end{array}\right]=0$, and so $F_{u,u}=0$. Thus, $(E_0)_{u,u}=\frac{1}{2}$, and applying Theorem \ref{thm:sed}(2) yields 2b. 

We now prove 3. For $\ell \in\{2,\ldots, f\}$, let $\bv_{\ell-1}$ be the vector whose $j$th entry is equal to 1 for each vertex $j$ in the part of size $n_1$ in $H_1$, equal to $-1$ for each vertex $j$ in the part of size $n_\ell$ in $H_\ell$, equal to $\frac{\sqrt{n_1}}{\sqrt{m_1}}$ for each vertex $j$ in the part of size $m_1$ in $H_1$, equal to $-\frac{\sqrt{n_\ell}}{\sqrt{m_\ell}}$ for each vertex $j$ in the part of size $m_\ell$ in $H_\ell$ and equal to 0 otherwise. Then $\{\bv_1,\ldots,\bv_{f-1}\}$ is a linearly independent set of eigenvectors associated with the eigenvalue $\sqrt{e}$, each of which has a zero at the entry indexed by $v$. This implies that $\sqrt{e}\notin\supp_{\be_v}$. Consequently, $\supp_{\be_v}=\{\pm\lambda, 0\}$. Now, from the proof of 2, we get that $(E_0)_{v,v}=(\bv\bv^T)_{v,v}=\frac{1}{1+\sum_{j=1}^f\frac{1}{n_j}}$. If $\sum_{j=1}^f\frac{1}{n_j}>1$, then $(E_0)_{v,v}<\frac{1}{2}$, and so $v$ is not sedentary in $G$ by Corollary \ref{cor18}(1). Now, if $\sum_{j=1}^f\frac{1}{n_j}=1$, then $(E_0)_{v,v}=\frac{1}{2}$. Since vertex $v$ is periodic, Theorem \ref{thm:sed}(2) implies that $v$ is not sedentary in $G$. Lastly, if $\sum_{j=1}^f\frac{1}{n_j}<1$, then $(E_0)_{v,v}>\frac{1}{2}$, and so $v$ is sedentary in $G$ by Theorem \ref{thm:sed}(1). This establishes 3.

Finally, we prove 4. If $u\in B_1$, then from the proof of 2, we have $0\in\supp_{\be_u}$. In particular, since $u$ belongs to a part of size $1$, we get $n_\ell=1$ for some $\ell$, and so $(E_0)_{u,u}=\frac{1}{1+\sum_{j\neq \ell}\frac{1}{n_j}+1}<\frac{1}{2}$. By Corollary \ref{cor18}(2), 4a holds. Meanwhile, if $u\in B_2$, then, from the proof of 2, $0\notin\supp_{\be_u}$; and from the proof of 3, $\pm\sqrt{e}\in\spec(G)$. Thus, $\supp_{\be_u}=\{\pm\lambda,\pm\sqrt{e}\}$, and so $u$ is not sedentary by Theorem \ref{thm:bipartite}. This proves 4b.
\end{proof}

If $e$ is odd and $H$ has no component isomorphic to a star, then all but one vertex of $G$ is sedentary by Theorem \ref{21}(1,3). We now give an example where $G$ has more than one vertex that is not sedentary.

\begin{example}
Let $G$ be the graph constructed above, where $H$ is a union of $K_{1,4}$ and $K_{2,2}$ (so $e=4$ and $f=2$), see Figure \ref{yay}. Then $\spec(G)=\left\{\sqrt{7},2,0^{6},-2,-\sqrt{7}\right\}.$ Blue and pink vertices are sedentary by Theorem~\ref{21}(1) and Theorem~\ref{21}(2a), respectively. However, the white vertex with degree five is not sedentary by Theorem~\ref{21}(4b). Moreover, the white vertices with degree three are also not sedentary by Theorem \ref{21}(2b), because $\lambda=\sqrt{7}$ and $\sqrt{e}=2$ are linearly independent over $\mathbb{Q}$. Since $n_1=4$ and $n_2=2$, we have $\frac{1}{n_1}+\frac{1}{n_2}<1$, and so the black vertex is sedentary by Theorem~\ref{21}(3).
\end{example}

We note that there are other families of bipartite graphs with five distinct eigenvalues. In \cite{li2024bipartite}, the authors characterized the family of unweighted bipartite graphs with eigenvalues $\{r,1^{(a)},0^{(b)},-1^{(a)},-r\}$ for some $r>1$. They showed that a graph in this family belongs to either one of five infinite families or a list of seven sporadic graphs. It would be interesting to characterize vertex sedentariness in these graphs. We suspect that the proof will employ similar techniques used in Corollary \ref{subdstar} and Theorem \ref{21}. In fact, subdivided stars are a subfamily of one of the five infinite families given in \cite{li2024bipartite}.

Finally, we characterize vertex sedentariness in trees with diameter at most three. Note that a tree with diameter at most two is a star, in which case the following result is straightforward from Proposition \ref{kmn}.

\begin{proposition}
Let $n\geq 3$. In the unweighted $K_{1,n}$, each leaf is sedentary but the central vertex is not.
\end{proposition}

The above result complements Corollary \ref{lem20}(2) which states that for $n\geq 3$, at least two leaves in $K_{1,n}$ are sedentary for any assignment of edge weights.

The diameter three trees are precisely the double stars $S_{k,\ell}$. These are graphs obtained from $K_2$ by attaching $k$ and $\ell$ pendent vertices to its two vertices.

\begin{theorem}
\label{skl}
Consider the unweighted double star $S_{k,\ell}$ with $k \geq \ell$. Let vertex $u$ be degree $k+1$, vertex $v$ degree $\ell+1$ and $w$ be any leaf adjacent to vertex $u$.
\begin{enumerate}
    \item For all $k,\ell\geq 1$, vertices $u$ and $v$ are not sedentary in $S_{k,\ell}$.
    \item Vertex $w$ is sedentary in $S_{k,\ell}$ if and only if $k\geq 3$ or $k=\ell=2$.
\end{enumerate}
\end{theorem}

\begin{proof}
For all $k,\ell\geq 1$, one checks that
\begin{equation*}
\spec(S_{k,\ell})=\left\{0^{(k+\ell-2)},\pm \frac{1}{\sqrt{2}}\sqrt{k+\ell+1\pm\sqrt{(\ell-k+1)^2+4k}}\right\}.
\end{equation*}
Since $\{\be_a-\be_b:a,b\ \text{leaves adjacent to $u$}\}\cup\{\be_c-\be_d:c,d\ \text{leaves adjacent to $v$}\}$ is a linearly independent set of $k+\ell-2$ eigenvectors for the eigenvalue 0 of $A(S_{k,\ell})$, we conclude that $0$ does not belong to either $\supp_{\be_u}$ or $\supp_{\be_v}$. Thus, vertices $u$ and $v$ are not sedentary in $S_{k,\ell}$ by Theorem \ref{thm:bipartite}. This proves 1. To prove 2, we proceed with cases. The cases $k\geq 3$ and $k=2$ follow from \cite[Corollary 43]{Monterde2023}. Now if $k=1$, then one can checks that $0\notin\supp_{\be_w}$. So Theorem \ref{thm:bipartite} implies that $w$ is not sedentary.
\end{proof}

Finally, we close this section by remarking that the unweighted subdivided star in Corollary \ref{subdstar} is a twin-free tree with diameter four that has no sedentary vertices. We suspect that for unweighted trees with diameter four, sedentariness generally occurs at twin vertices.

\section{Paths and cycles}\label{Sec:paths}

In this section, we show that unweighted paths and unweighted even cycles are not sedentary.

\begin{theorem}
\label{paths}
For all $n\geq 3$, the unweighted $P_{n}$ is not sedentary.
\end{theorem}

\begin{proof}
Label the vertices of $P_n$ by $\{1,\dots, n\}$, with consecutive vertices adjacent.
It suffices to consider the case when $n$ is odd by Corollary \ref{matching}. If $u$ is an even vertex in $P_{n}$, then $0\notin\supp_{\be_u}$, so the desired result is obtained by applying Theorem \ref{thm:bipartite}. On the other hand, if $u$ is an odd vertex in $P_n$, then Corollary 10.2.3 and Lemma 10.2.5 in \cite{Coutinho2021} gives us
\begin{equation*}
U(t)_{u,u}=\frac{2}{n+1}+\frac{4}{n+1}\sum_{j=1}^{\frac{n-1}{2}}\cos\left(2 t \cos\left(\frac{j\pi}{n+1}\right) \right)\sin^2\left(\frac{uj\pi}{n+1}\right).
\end{equation*}
So at $t=\frac{\pi}{\sqrt{2}}$, we have
\begin{equation}
\label{patheq}
\begin{split}
\frac{n+1}{4}U\left(\frac{\pi}{\sqrt{2}}\right)_{u,u}&=\frac{1}{2}+\sum_{j=1}^{\frac{n-1}{2}}\cos\left(\sqrt{2}\pi\cos\left(\frac{j\pi}{n+1}\right)\right)\sin^2\left(\frac{uj\pi}{n+1}\right).
\end{split}
\end{equation}
It is easy to check that $P_n$ is not sedentary for $n\in\{3,5,7\}$. Hence, suppose $n\geq 9$. First, assume that $\frac{n+1}{4}$ is an integer. Let $A_j=\cos\left(\sqrt{2}\pi\cos\left(\frac{j\pi}{n+1}\right)\right)$ and $B_j=\cos\left(\sqrt{2}\pi\sin\left(\frac{j\pi}{n+1}\right)\right)$. Since $u$ is odd and $A_{\frac{n+1}{2}-j}=B_j$, we may write  (\ref{patheq}) as
\begin{equation*}
\begin{split}
\frac{n+1}{4}U\left(\frac{\pi}{\sqrt{2}}\right)_{u,u}
&=\frac{1}{2}+\bigg(\sum_{j=1}^{\frac{n+1}{4}-1}A_j\sin^2\left(\frac{uj\pi}{n+1}\right)\bigg)+\cos\pi\sin^2\left(\frac{u\pi}{4}\right)+\bigg(\sum_{j=\frac{n+1}{4}+1}^{\frac{n-1}{2}}A_j\sin^2\left(\frac{uj\pi}{n+1}\right)\bigg)\\
&=\sum_{j=1}^{\frac{n-3}{4}}A_j\sin^2\left(\frac{uj\pi}{n+1}\right)+\sum_{j=1}^{\frac{n-3}{4}}A_{\frac{n+1}{2}-j}\sin^2\left(\frac{u(\frac{n+1}{2}-j)\pi}{n+1}\right)\\
&=\sum_{j=1}^{\frac{n-3}{4}}A_j\sin^2\left(\frac{uj\pi}{n+1}\right)+B_j\cos^2\left(\frac{uj\pi}{n+1}\right)=\sum_{j=1}^{\frac{n-3}{4}}(A_j-B_j)\sin^2\left(\frac{uj\pi}{n+1}\right)+\sum_{j=1}^{\frac{n-3}{4}}B_j.
\end{split}
\end{equation*}
Since $\cos\left(\sqrt{2}\pi\cos x\right)-\cos\left(\sqrt{2}\pi\sin x\right)\leq 0$ for all $0\leq x\leq \frac{\pi}{4}$, we get that $A_j-B_j\leq 0$ for each $j$. Therefore, $\sum_{j=1}^{\frac{n-3}{4}}\left((A_j-B_j)\sin^2\left(\frac{uj\pi}{n+1}\right)\right)\leq 0$. In order to prove that $\frac{n+1}{4}U\left(\frac{\pi}{\sqrt{2}}\right)_{u,u}\leq 0$, it suffices to show that $\sum_{j=1}^{\frac{n-3}{4}}B_j\leq 0$. Let $\zeta=\lfloor\frac{n+1}{\pi}\arcsin(\frac{1}{2\sqrt{2}})\rfloor\approx\lfloor 0.11527(n+1)\rfloor$. Since $n\geq 9$, we have $\zeta\geq 1$. Moreover, $B_j\leq 0$ whenever $j\geq \zeta+1$. Now, observe that
\begin{equation*}
\begin{split}
\sum_{j=1}^{\frac{n-3}{4}}B_j
&=\sum_{j=1}^{\zeta}B_j+\sum_{j=\zeta+1}^{\frac{n-3}{4}-\zeta}B_j+\sum_{j=\frac{n-3}{4}-\zeta+1}^{\frac{n-3}{4}}B_j\\
&=\sum_{j=1}^{\zeta}\cos\left(\sqrt{2}\pi\sin\left(\frac{j\pi}{n+1}\right)\right)+\sum_{j=\zeta+1}^{\frac{n-3}{4}-\zeta}B_j+\sum_{j=1}^{\zeta}\cos\left(\sqrt{2}\pi\sin\left(\frac{(\frac{n-3}{4}-j+1)\pi}{n+1}\right)\right)\\
&=\sum_{j=1}^{\zeta}\left(\cos\left(\sqrt{2}\pi\sin\left(\frac{j\pi}{n+1}\right)\right)+\cos\left(\sqrt{2}\pi\sin\left(\frac{\pi}{4}-\frac{j\pi}{n+1}\right)\right)\right)+\sum_{j=\zeta+1}^{\frac{n-3}{4}-\zeta}B_j.
\end{split}
\end{equation*}
Since $\sum_{j=\zeta+1}^{\frac{n-3}{4}-\zeta}B_j\leq 0$ and $\cos\left(\sqrt{2}\pi\sin x\right)+\cos\left(\sqrt{2}\pi\sin\left(\frac{\pi}{4}-x\right)\right)\leq 0$ for all $x\in\R$, it follows that $\sum_{j=1}^{\frac{n-3}{4}}B_j\leq 0$. Hence $\frac{n+1}{4}U\left(\frac{\pi}{\sqrt{2}}\right)_{u,u}\leq 0$. Now, suppose $\frac{n+1}{4}$ is not an integer. Equivalently, $\frac{n-1}{4}$ is an integer because $n$ is odd. Using the same argument above, we may write (\ref{patheq}) as
\begin{equation}
\label{eq:new}
\begin{split}
\frac{n+1}{4}U\left(\frac{\pi}{\sqrt{2}}\right)_{u,u}
&=\frac{1}{2}+\sum_{j=1}^{\frac{n-1}{4}}(A_j-B_j)\sin^2\left(\frac{uj\pi}{n+1}\right)+\sum_{j=1}^{\frac{n-1}{4}}B_j.
\end{split}
\end{equation}
The same argument yields $\sum_{j=1}^{\frac{n-1}{4}}B_j<0$. Finally, since $A_j-B_j\leq 0$ for each $j$ and $A_1-B_1<-0.66$, (\ref{eq:new}) gives us $U\left(\frac{\pi}{\sqrt{2}}\right)_{u,u}\leq 0$. Combining the two cases, we get that any odd vertex in $P_n$ is not sedentary.
\end{proof}

From Theorem \ref{paths}, all unweighted odd paths are not sedentary. This complements Proposition \ref{oddpath} which states that any odd path has a weighting that results in an end vertex being sedentary.

\begin{theorem}
\label{cn}
For all even $n\geq 4$, the unweighted $C_{n}$ is not sedentary.
\end{theorem}

\begin{proof}
Let $n=2k$. For odd $k$, $C_k$ is nonsingular. So Corollary \ref{cnodd} yields the desired conclusion. Suppose $k$ is even. Then $\lambda_{\frac{n}{4}}=0$ is an eigenvalue of $A(C_n)$ with multiplicity two.
For any vertex $u$ of $C_n$ and for all $t\in\R$, we have
\begin{equation*}
\displaystyle U(t)_{u,u}=(2/n)+(2/n)\cos(2t)+(4/n)\sum_{1}^{(n/4)-1}\cos(\lambda_j t).
\end{equation*}
When $t=\frac{\pi}{2}$, we obtain
\begin{equation}
\label{eqCn}
(n/4) U(\pi/2)_{u,u}=\sum_{j=1}^{(n/4)-1}\cos(\lambda_j(\pi/2)).
\end{equation}
Note that $\lambda_j\pi/2=\cos(2j\pi/n)\pi$. For $0<j<\frac{n}{4}$, we have $\cos(\lambda_j\pi/2)\leq 0$ if and only if $j\leq \lfloor\frac{n}{6}\rfloor$. Now, each $j\in\{\lfloor\frac{n}{6}\rfloor+1,\ldots,\frac{n}{4}-1\}$ can be written as $\frac{n}{4}-\ell$, where $\ell\in\{1,\ldots,\frac{n}{4}-\lfloor\frac{n}{6}\rfloor-1\}$. Thus, if $j>\lfloor\frac{n}{6}\rfloor$, then  $\lambda_j(\pi/2)=\pi\cos(2j\pi/n)=\pi\cos\left(\frac{2(n/4-\ell)\pi}{n}\right)=\pi\cos\left(\frac{\pi}{2}-\frac{2\ell\pi}{n}\right)=\pi\sin(2\ell\pi/n).$
This allows us to write (\ref{eqCn}) as
\begin{equation*}
\begin{split}
\frac{n}{4}U\left( \pi/2 \right)_{u,u}&=\sum_{j=1}^{\lfloor n/6\rfloor}\cos(\lambda_j(\pi/2))+\sum_{j=\lfloor n/6\rfloor+1}^{(n/4))-1}\cos(\lambda_j(\pi/2))\\
&=\sum_{j=1}^{\lfloor n/6\rfloor}\cos(\pi\cos(2j\pi/n))+\sum_{\ell=1}^{(n/4)-\lfloor n/6\rfloor-1}\cos(\pi\sin(2\ell\pi/n))\\
&=\sum_{j=1}^{(n/4)-\lfloor n/6\rfloor-1}[\cos(\pi\cos(2j\pi/n))+\cos(\pi\sin(2j\pi/n))]+\sum_{j=(n/4)-\lfloor n/6\rfloor}^{\lfloor n/6\rfloor}\cos(\pi\cos(2j\pi/n)).
\end{split}
\end{equation*}
Since $\cos(\pi\cos x)+\cos(\pi\sin x)\leq 0$ for all $x\in\R$ and $\cos(\pi\cos x)\leq 0$ for all $x\in[0,\pi/3]$, the above equation implies that $U(\pi/2)_{u,u}\leq 0$. That is, each vertex $u$ of $C_n$ is not sedentary.
\end{proof}

Since odd cycles are nonbipartite, our results in Section \ref{sec:bip} do not apply to this family. To complement the fact that $C_5$ is not sedentary, we make the following observation about a weighted $C_3$.

\begin{example}
Let $V(X)=\{1,2,3\}$ with edges $\{1,2\}$, $\{2,3\}$ and $\{1,3\}$ with corresponding weights $1$, $1$ and $\alpha\in\R$, respectively. 
Then 
\begin{equation*}
\spec(X)=\left\{\theta=-\alpha,\lambda=\frac{1}{2}(\alpha-\sqrt{\alpha^2+8}),\eta=\frac{1}{2}(\alpha+\sqrt{\alpha^2+8})\right\}.
\end{equation*}
If $\alpha>0$, then  $(E_{\lambda})_{2,2}>\frac{1}{2}$, while if $\alpha<0$, then $(E_{\eta})_{2,2}>\frac{1}{2}$. Thus, if $\alpha\neq 0$, then vertex $2$ is sedentary.
\end{example}

Next, we complement Theorem \ref{cn} by providing a weighting for $C_{4k}$ that yields a sedentary vertex.

\begin{proposition}
\label{c4k}
Let $u,v,w$ be vertices of a weighted $C_{4k}$ such that $u$ and $v$ are adjacent, $v$ and $w$ are adjacent, and $\{u,v\}$ and $\{v,w\}$ have edge weights $\frac{1}{\alpha}$. If $|\alpha|>\sqrt{2k-1}$ and all other edge weights are one, then vertex $v$ is sedentary.  
\end{proposition}

\begin{proof}
Let $X$ be the weighted $C_{4k}$ in the assumption, and order the vertices so consecutive ones are adjacent. If we index the first three rows of $A(X)$ by $u,v,w$ respectively, then 
\begin{equation*}
\frac{1}{\sqrt{2k} }[1,0,-1,0,1,0,-1,0, \dots ]^T\quad \text{and} \quad \frac{1}{\sqrt{\alpha^2+ 2k-1}}[0,\alpha,0,-1,0,1,0,-1, \dots ]^T
\end{equation*}
are orthonormal eigenvectors associated with the eigenvalue 0 of $A(X)$. Thus, $(E_0)_{2,2}\geq \frac{\alpha^2}{\alpha^2+2k-1}$, and so $(E_0)_{2,2}>\frac{1}{2}$ whenever $|\alpha|>\sqrt{2k-1}$. Invoking Theorem \ref{thm:sed}(1) completes the proof.
\end{proof}

We close this section with an example.

\begin{example}
\label{ex1}
Let $G$ be a weighted $C_3$ with edge weights  $\alpha,\beta,\gamma\neq 0$ and $H$ be a weighted $C_4$ with edge weights  $\alpha,\beta,\gamma,\delta\neq 0$. The following hold for the graphs $X,X',Y$ and $Y'$ in Figure \ref{fig:cycle}.
\begin{enumerate}
\item Since $X=K_2\times G$ and $G$ is nonsingular for all $\alpha,\beta,\gamma\neq 0$, $X$ is not sedentary by Theorem \ref{bipdoub}. On the other hand, since $X'=S(G)$, Corollary \ref{c6} implies that $X'$ is not sedentary. 
\item Note that $Y$ is nonsingular if and only if $|\alpha\gamma|\neq |\beta\delta|$. Since $Y$ is bipartite, each vertex of $Y$ is not sedentary whenever $|\alpha\gamma|\neq |\beta\delta|$ by Corollary \ref{cor:intbipgraphs}. On the other hand, $Y'=S(H)$ and $A(Y')$ has nullity at least two for all $\alpha,\beta,\gamma,\delta\neq 0$. If $\beta=\gamma=\delta=1$, then by Proposition \ref{c4k}, the vertex of $Y'$ incident to the edges with weight $\alpha$ is sedentary when $|\alpha|<\frac{1}{\sqrt{3}}$.
\end{enumerate}
\end{example}

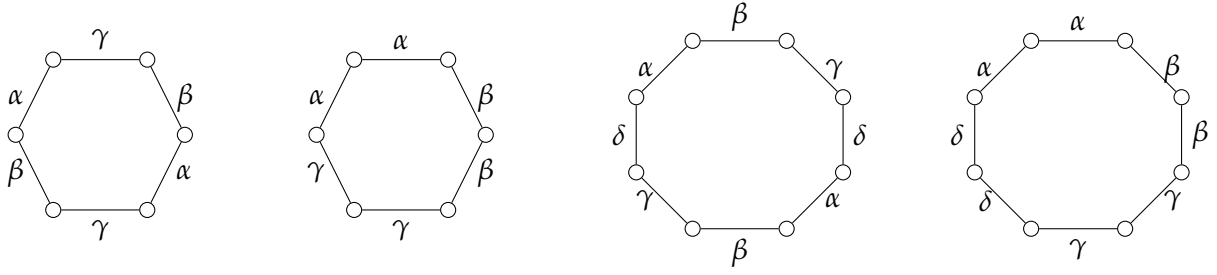
\begin{figure}
    \centering
    \begin{tikzpicture}
    		\tikzset{enclosed/.style={draw, circle, inner sep=0pt, minimum size=.2cm}}
		\node[enclosed] (la1) at (0.75,1) {};
		\node[enclosed] (lb1) at (0.25,0) {};
		\node[enclosed] (lc1) at (0.75,-1) {};
		\node[enclosed] (la2) at (2,1) {};
		\node[enclosed] (lb2) at (2.5,0) {};
		\node[enclosed] (lc2) at (2,-1) {};
                 \draw (la1)--node[left] {$\alpha$}(lb1);
                 \draw (lb1)--node[left] {$\beta$}(lc1);
                 \draw (la2)--node[right] {$\beta$}(lb2);
                 \draw (lb2)--node[right] {$\alpha$}(lc2);
                 \draw (la1)--node[above] {$\gamma$}(la2);
                 \draw (lc1)--node[below] {$\gamma$}(lc2);
    
		\node[enclosed] (a1) at (4.75,1) {};
		\node[enclosed] (b1) at (4.25,0) {};
		\node[enclosed] (c1) at (4.75,-1) {};
		\node[enclosed] (a2) at (6,1) {};
		\node[enclosed] (b2) at (6.5,0) {};
		\node[enclosed] (c2) at (6,-1) {};
                 \draw (a1)--node[left] {$\alpha$}(b1);
                 \draw (b1)--node[left] {$\gamma$}(c1);
                 \draw (a2)--node[right] {$\beta$}(b2);
                 \draw (b2)--node[right] {$\beta$}(c2);
                 \draw (a1)--node[above] {$\alpha$}(a2);
                 \draw (c1)--node[below] {$\gamma$}(c2);

        \node[enclosed] (m1) at (8.5,0.5) {};
		\node[enclosed] (m2) at (9.25,1.25) {};
		\node[enclosed] (m3) at (10.5,1.25) {};
		\node[enclosed] (m4) at (11.25,0.5) {};
		\node[enclosed] (m5) at (11.25,-0.5) {};
        \node[enclosed] (m6) at (10.5,-1.25) {};
        \node[enclosed] (m7) at (9.25,-1.25) {};
        \node[enclosed] (m8) at (8.5,-0.5) {};
        \draw (m1)--node[left] {$\alpha$}(m2);
        \draw (m2)--node[above] {$\beta$}(m3);
        \draw (m3)--node[right] {$\gamma$}(m4);
        \draw (m4)--node[right] {$\delta$}(m5);
        \draw (m5)--node[right] {$\alpha$}(m6);
        \draw (m6)--node[below] {$\beta$}(m7);
        \draw (m7)--node[left] {$\gamma$}(m8);
        \draw (m8)--node[left] {$\delta$}(m1);

        \node[enclosed] (n1) at (13,0.5) {};
		\node[enclosed] (n2) at (13.75,1.25) {};
		\node[enclosed] (n3) at (15,1.25) {};
		\node[enclosed] (n4) at (15.75,0.5) {};
		\node[enclosed] (n5) at (15.75,-0.5) {};
        \node[enclosed] (n6) at (15,-1.25) {};
        \node[enclosed] (n7) at (13.75,-1.25) {};
        \node[enclosed] (n8) at (13,-0.5) {};
        \draw (n1)--node[left] {$\alpha$}(n2);
        \draw (n2)--node[above] {$\alpha$}(n3);
        \draw (n3)--node[right] {$\beta$}(n4);
        \draw (n4)--node[right] {$\beta$}(n5);
        \draw (n5)--node[right] {$\gamma$}(n6);
        \draw (n6)--node[below] {$\gamma$}(n7);
        \draw (n7)--node[left] {$\delta$}(n8);
        \draw (n8)--node[left] {$\delta$}(n1);
                 
    \end{tikzpicture}
    \caption{The graphs $X$ (leftmost) and $X'$ (center left), which are weighted versions of $C_6$, and the graphs $Y$ (center right) and $Y'$ (rightmost), which are weighted versions of $C_8$}
    \label{fig:cycle}
\end{figure}

\section{More constructions}\label{Sec:cons1}

In this penultimate section, we give two constructions of graphs where it is possible to track the vertices that are sedentary.

\subsection*{Cartesian product}

Let $X$ and $Y$ be weighted graphs. The \textit{Cartesian product} of $X$ and $Y$, denoted $X\square Y$, is the graph with vertex set $V(X)\times V(Y)$ and adjacency matrix 
\begin{equation*}
A(X)\otimes I+I\otimes A(Y),
\end{equation*}
where $A\otimes B$ denotes the Kronecker product of matrices $A$ and $B$. From~\cite[Lemma 4.2]{Godsil2012a}, we have
\begin{equation*}
U_{X\square Y}(t)=U_X(t)\otimes U_Y(t)\quad \text{for all $t\in\R$}.
\end{equation*}
Thus, if $u\in V(X)$ and $y\in V(Y)$, then
\begin{equation}
\label{cartprodeq}
|U_{X\square Y}(t)_{(u,v),(u,v)}|=|U_{X}(t)_{u,u}|\cdot|U_{Y}(t)_{v,v}|\quad \text{for all $t\in\R$}.
\end{equation}
The following is immediate from (\ref{cartprodeq}).

\begin{theorem}
\label{cp}
Vertex $(u,v)$ is sedentary in $X\square Y$ if and only if vertices $u$ and $v$ are sedentary in $X$ and $Y$, respectively.
\end{theorem}

Using our results in the previous sections, we can build bigger graphs using the Cartesian product operation by virtue of Theorem \ref{cp}. For example, Corollary \ref{cor:intbipgraphs} and Theorem \ref{thm:bipeveneval} combined with Theorem \ref{cp}, implies that if $X$ or $Y$ is a nonsingular weighted bipartite graph, then $X\square Y$ is not sedentary. In particular, if $X$ or $Y$ is a weighted bipartite graph with unique perfect matching, then $X\square Y$ is not sedentary.

Next, since $K_n$ is sedentary if and only if $n\geq 3$, the following is immediate from Theorem \ref{cp}.

\begin{corollary}
The Cartesian product of unweighted complete graphs is sedentary if and only if none of its factors is isomorphic to $K_2$.
\end{corollary}

The next result shows that the grid and the torus are not sedentary. It is immediate from Theorems \ref{cp}, \ref{paths}, \ref{cn} and Corollary \ref{matching}.

\begin{corollary}
For all integers $n_1,\ldots,n_k\geq 2$, each vertex in the unweighted $P_{n_1}\square\cdots\square P_{n_k}$ is not sedentary. Moreover, for all even integers $n_1,\ldots,n_k\geq 2$, we have that
\begin{enumerate}
\item each vertex in the unweighted $C_{n_1}\square\cdots\square C_{n_k}$ is not sedentary, and
\item each vertex $P_{n_1}\square\cdots\square P_{n_k}$ is not sedentary for any assignment of edge weights.
\end{enumerate}
\end{corollary}

The \textit{Hamming graph} $H(d,q)$ is a Cartesian product of $d$ copies of unweighted $K_q$. The \textit{Doob graph} $D(\ell,d)$ is the Cartesian product of $\ell$ copies of the unweighted Shrikhande graph and one copy of $H(d,4)$. Hamming graphs are known to be sedentary \cite{Monterde2023}. Below we show that Doob graphs are also sedentary.

\begin{corollary}
$H(d,q)$ is $\big(1-\frac{2}{q}\big)^d$-sedentary for all $d\geq 1$ and $q\geq 3$ with equality at $t=\frac{\pi}{q}$. Moreover, $D(\ell,d)$ is $\big(\frac{1}{2^{2\ell+d}}\big)$-sedentary for all $d,\ell\geq 1$, with equality at $t=\frac{\pi}{q}$. 
\end{corollary}

\begin{proof}
Since $K_q$ is $(1-\frac{2}{q})$-sedentary for all $q\geq 3$ with equality at $t=\frac{\pi}{q}$, making use of (\ref{cartprodeq}) yields the first statement. One checks that the Shrikhande graph is $(\frac{1}{4})$-sedentary with equality at $t=\frac{\pi}{4}$. As $H(d,4)$ is $\big(\frac{1}{2^d}\big)$-sedentary with equality at $t=\frac{\pi}{4}$, again applying (\ref{cartprodeq}) yields the second statement.
\end{proof}

\subsection*{Rooted product}
Let $X$ be a weighted graph with vertices $\{v_1,\ldots,v_n\}$ and $Y$ be a disjoint union of weighted graphs $Y_1,\ldots,Y_n$, rooted at $y_1,\ldots,y_n$, respectively. The \textit{rooted product} $X(Y)$ of $X$ and $Y$ is the weighted graph obtained by identifying each vertex $v_i$ of $X$ with vertex $y_i$ of $Y_i$. The rooted product and its spectra was first studied in \cite{godsil1978new}.

If $X$ is a singular weighted bipartite graph, then it is possible that $X$ contains a sedentary vertex, as demonstrated by weighted odd paths in Proposition \ref{oddpath}. The following result uses rooted products to generate a (bipartite) graph with no sedentary vertices from any given bipartite graph. 

\begin{theorem}
\label{biprp}
Suppose $X$ is bipartite and $Y$ is a disjoint union of bipartite graphs $Y_1,\ldots,Y_n$, each with a unique perfect matching $M_i$ and rooted at vertices $y_1,\ldots,y_n$. Then $X(Y)$ is not sedentary for any assignment of edge weights.
\end{theorem}

\begin{proof}
By assumption, $X(Y)$ is bipartite with a unique perfect matching $\bigcup_i M_i$. Applying Theorem \ref{thm:bipeveneval}  yields the desired conclusion.
\end{proof}

A \textit{graph with tails} is a graph with paths attached to some of its vertices. Note that $X(Y)$ is a graph with tails whenever $Y$ is a union of paths. In this case, applying Theorem \ref{biprp} to $X(Y)$ yields the next example.

\begin{example}
\label{62}
If $X$ is bipartite and $Y$ is a union of even paths, then $X(Y)$ is a graph with tails that is not sedentary for any assignment of edge weights. 
\end{example}

In Example \ref{62}, if $X$ is a cycle, then $X(Y)$ is a unicyclic graph that is nonsingular (because $X(Y)$ has a unique perfect matching), $X(Y)$ is not sedentary for any assignment of edge weights, and $X(Y)$ is not a subdivision graph. This complements our result in Corollary \ref{c6} on subdivision graphs $S(X)$ with $C_6$ as a subgraph that are nonsingular and not sedentary for any assignment of edge weights.

\section{Future work}\label{Sec:open}

In this paper, we reported some recent progress on vertex sedentariness in quantum walks, with emphasis on bipartite graphs, trees, and planar graphs. In particular, we showed that almost all trees and almost all planar graphs contain sedentary vertices. This is in contrast to nonsingular bipartite graphs, which do not contain any sedentary vertex. In order to inspire future work on this topic, we present some open questions that deserve further exploration.

\begin{enumerate}
\item Provide a characterization of sedentariness that only involves the eigenvalues and orthogonal projection matrices of the adjacency matrix. For periodic vertices with integer eigenvalue supports, we suspect that this maybe unlikely since by virtue of Lemma \ref{lem:persed}, nonsedentariness of a vertex is equivalent to finding a root on the unit complex circle of the following polynomial over $\R$:
\begin{equation*}
\sum_{\lambda\in\supp_{\be_u}}a_\lambda z^{\lambda-\lambda_0},
\end{equation*}
where $\lambda_0=\min_{\lambda\in \supp_{\be_u}}\lambda$. The above polynomial was derived from $U(t)_{u,u}=\sum_{\lambda\in\supp_{\be_u}}e^{it\lambda}(E_\lambda)_{u,u}$ by letting $z:=e^{it}$, $a_\lambda:=(E_\lambda)_{u,u}$, and  factorizing $e^{it\lambda_0}$. Another idea that may lead to a characterization is the concept of pretty good group state transfer (PGGST) first introduced in \cite[Section 10]{monterde2023new}. Proposition 2 in this paper showed that a vertex $u$ in a graph $X$ is either sedentary or involved in PGGST with a nonempty subset of $V(X)\backslash\{u\}$. Thus, a characterization of PGGST in terms of the eigenvalues and orthogonal projection matrices of the adjacency matrix will naturally lead to a characterization of vertex sedentariness.

\item We are interested in a complete characterization of sedentariness in unweighted cycles. By virtue of Theorem \ref{cn}, it suffices to ask: which unweighted odd cycles are sedentary? More generally, it is desirable to characterize all weightings of cycles and odd paths that produce sedentary vertices.

\item Two natural questions arise from Theorem \ref{thm:bipeveneval}. First, is there an infinite family of nonbipartite graphs with a unique perfect matching that is sedentary for any assignment of edge weights? And second, amongst all unweighted bipartite graphs with at least two perfect matchings, which ones are nonsingular? How about those which are sedentary?

\item We also pose a question related to the previous one. Given a connected bipartite graph $X$ with $m\geq 2$ perfect matchings, what is the least number of edges that must be removed from $X$ in order to obtain a connected graph $X'$ with a unique perfect matching?

\item Theorem \ref{planar} motivates us to determine infinite families of planar graphs that are not sedentary. In particular, amongst all planar graphs whose every edge belongs to a cycle (such as maximal planar graphs and outerplanar graphs), which ones are not sedentary? 

\item Lastly, Theorems \ref{alltrees} and \ref{skl} encourage us to further investigate vertex sedentariness on trees. In particular, are twin-free trees generally not sedentary? We also seek a characterization of vertex sedentariness in trees with diameter four or more. In \cite{kirkland2025state}, PGST was characterized for trees with diameter four. Since PGST and vertex sedentariness are mutually exclusive, it would be interesting to determine whether vertices in such trees that are not involved PGST are sedentary.
\end{enumerate}

\noindent
\emph{Acknowledgement}: 
K.~Meagher is supported by NSERC Discovery Grant RGPIN-2025-04101. H.~Monterde is supported by the Pacific Institute for the Mathematical Sciences through the PIMS-Simons Postdoctoral Fellowship. We are grateful to the referees for the comments that improved the paper. We also thank David Feder, Alastair Kay, Christino Tamon, and Tom Wong for discussions on the potential applications of vertex sedentariness especially to quantum memory.

\bibliographystyle{alpha}
\bibliography{mybibfile}
\end{document}